\documentclass[10pt]{amsart}
\usepackage{amsmath}
\usepackage{amsfonts}
\usepackage{amsthm}
\usepackage{amssymb}
\usepackage[all]{xy}
 \SelectTips{cm}{}
\usepackage{hyperref}

\title{Local Framings}
\author{David Barnes}
\address{David Barnes \\ The University of Sheffield \\ School of Mathematics and Statistics \\ Hicks Building \\  Sheffield \\ S3 7RH \\ UK
                }
\email{D.J.Barnes@sheffield.ac.uk}
\thanks{The first author was supported by EPSRC grant EP/H026681/1} 
\author{Constanze Roitzheim}
\address{Constanze Roitzheim \\ University of Glasgow \\ Department of Mathematics \\ 15 University Gardens \\ Glasgow \\ G12 8QW \\ UK
        }
\email{constanze.roitzheim@glasgow.ac.uk}
\thanks{The second author was supported by EPSRC grant EP/G051348/1.}
\date{$20^\text{th}$ July 2011}

\DeclareMathOperator{\ev}{ev}
\DeclareMathOperator{\id}{id}
\DeclareMathOperator{\Ho}{Ho}

\DeclareMathOperator{\Hom}{Hom}
\DeclareMathOperator{\Map}{Map}
\DeclareMathOperator{\SSet}{sSet}
\DeclareMathOperator{\map}{map}

\DeclareMathOperator{\h}{H}
\DeclareMathOperator{\ch}{Ch}

\DeclareMathOperator{\Pic}{Pic}

\newcommand{\Sp}{\mathcal{S}}
\newcommand{\C}{\mathcal{C}}

\newcommand{\co}{\colon \!}
\newcommand{\lradjunction}{\,\,\raisebox{-0.1\height}{$\overrightarrow{\longleftarrow}$}\,\,}

\newcommand{\sset}{\SSet_*}

\newcommand{\ssetsmash}{\times}
\newcommand{\spectrasmash}{\wedge}

\newcommand{\Esmash}{\wedge^L_E}
\newcommand{\tensor}{\otimes}

\newcommand{\sstensor}{\wedge}
\newcommand{\Esframe}{\wedge^L_E}
\newcommand{\sframe}{\wedge}

\newtheorem{theorem}{Theorem}[section]
\newtheorem{proposition}[theorem]{Proposition}
\newtheorem{corollary}[theorem]{Corollary}
\newtheorem{lemma}[theorem]{Lemma}

\newtheorem{definition}[theorem]{Definition}
\newtheorem{ex}[theorem]{Example}

\newtheorem{rmk}[theorem]{Remark}

\newtheorem*{thm}{Theorem}

\begin{document}

\begin{abstract}

Framings provide a way to construct Quillen functors from simplicial sets to any given model category. A more structured set-up studies stable frames giving Quillen functors from spectra to stable model categories. We will investigate how this is compatible with Bousfield localisation to gain insight into the deeper structure of the stable homotopy category. We further show how these techniques relate to rigidity questions and how they can be used to study algebraic model categories. 

\end{abstract}

\maketitle

\section*{Introduction}

The two categories most important to homotopy theory are the stable homotopy category 
and the homotopy category of simplicial sets. 
It is very hard to study either of these categories, 
so a standard and highly successful method, known as Bousfield localisation, is often used. 
The idea is to look at `smaller pieces' of these categories.
These pieces have less information than the whole category, but are easier to work with as they are more structured.
To apply this method, one takes a homology theory $E_*$ and declares that
two simplicial sets (or two spectra) are equivalent if there is a map between them which induces an isomorphism
of $E_*$--homology. The resulting homotopy category is called the $E$--local homotopy category of simplicial sets
or the $E$--local stable homotopy category. 

There are many other model categories whose homotopy category behaves like 
a category of simplicial sets or spectra. 
The homotopy category of any pointed model category $\C$ is a 
closed module over the homotopy category of pointed simplicial sets, \cite{Hov99}. 
We show that this action extends to an action of the $E$--local
homotopy category of pointed simplicial sets if and only if the 
simplicial mapping spaces $\map(X,Y)$
are $E$--local simplicial sets for any $X$ and $Y$ in $\C$. 
We call such a model category $E$--familiar. 

If $\C$ is a pointed model category then there is a functor $\Sigma \co \Ho(\C) \to \Ho(\C)$, 
which corresponds to tensoring with the simplicial set $S^1$. If this functor is an
equivalence then the model category $\C$ is said to be stable. 
The work of \cite{Len11} shows that for stable $\C$, $\Ho(\C)$ has an action of the stable homotopy category. 
We have studied when this action is compatible with $E$--localisation 
and have the following characterisation of compatibility,
which is Theorem \ref{localmappingspectrum}.

\begin{thm}
If $\C$ is a stable model category, then the action 
of the stable homotopy category on $\Ho(\C)$ passes 
to an action of the $E$--local stable homotopy category 
if and only if the mapping spectra $\Map(X,Y)$
are $E$--local spectra for any $X$ and $Y$ in $\C$. 
\end{thm}

We call such a model category stably $E$--familiar. 
It is important to note that in general being stable and $E$--familiar
is not sufficient to be stably $E$--familiar.

As an application, we study how these techniques relate to rigidity of stable model categories. A stable model category is called rigid if its homotopical behaviour only depends on the triangulated structure of its homotopy category. The main examples are spectra themselves \cite{Sch07} and $K_{(2)}$--local spectra \cite{Roi07}. We show how the proofs of those results fit into our framework. This will provide a more streamlined formal setting for future rigidity proofs.

We also consider an alternative approach to rigidity, which investigates how much homotopical information is seen by framings. The answer, Theorem \ref{thm:modrigid}, is that in the case of a smashing localisation, the homotopical information of $E$--local spectra is entirely encoded in the $\Ho(\Sp)$--module structure of the $E$--local stable homotopy category. 

\begin{thm}
Let $L_E$ be a smashing localisation and let
\[
\Phi: \Ho(L_E \Sp) \longrightarrow \Ho(\C)
\]
be an equivalence of triangulated categories. Then the following are equivalent.
\begin{itemize}
\item $\Phi$ is the derived functor of a Quillen equivalence.
\item $\Phi$ is a $\Ho(\Sp)$--module functor.
\end{itemize}
\end{thm}

In \cite{SchShi02}, Schwede and Shipley show that a stable model category is entirely determined by the triangulated structure of its homotopy category together with a $\pi_*(\mathbb{S})$--action. For a stably $E$--familiar model category, we prove $E$--local analogues. This offers a technical advantage as the homotopy groups of the $E$--local spheres tend to be more highly structured and better understood than $\pi_*(\mathbb{S})$. 

A final application is examining algebraic model categories ($\ch(\mathbb{Z})$--model categories) that are also stably $E$--familiar. Our conclusion is Theorem \ref{thm:algrational}:

\begin{thm}
The model category of $E$--local spectra,
$L_E \Sp$, is an algebraic model category 
if and only if $E=\h\mathbb{Q}$.
\end{thm}

\subsection*{Organisation}
Section \ref{sec:Eloc} is a reminder of the notion
of Bousfield localisations of simplicial sets of spectra. Section \ref{sec:modsimp} recalls the notions of $\C$--module categories and $\C$--model categories. 
Section \ref{sec:framings} summarises Hovey's work on framings which proves that 
the homotopy category of any pointed model category is a $\Ho(\sset)$--module.

Section \ref{sec:Elocalframings} marks the start of the new work. We study 
when a framing on a model category is compatible with the $E$--local model structure on simplicial sets and define the notion of $E$--familiar model categories. In Section \ref{sec:Efamiliar} we study the properties of these $E$--familiar model categories. Furthermore, we show how our set-up generalises the notion of an $L_E \sset$--model category.

We move to a stable setting and use Lenhardt's notion of stable frames to replace simplicial sets with spectra in Section \ref{sec:Stableframes}. Following a similar pattern to the non-stable case,
in Section \ref{sec:Efamiliarstable},  we 
ask when are these stable frames compatible with the $E$--local model structure 
on spectra. 
The fact that a stable $E$--familiar model category is not, in general, a stably $E$--familiar model category is examined in  Section \ref{sec:versus}.
We finish the paper with examples and applications in Section \ref{sec:applications}. We start with some immediate consequences from the previous sections regarding chromatic localisations. The next part is dedicated to rigidity questions, followed by a study of $\pi_*(L_E \mathbb{S})$--actions. Finally, we can classify how 
$\Ho(L_E \Sp)$ acts on the homotopy category of a big class of algebraic stably $E$--familiar model categories.

\section{\texorpdfstring{$E$}{E}--localisations}\label{sec:Eloc}

Let $E$ be a spectrum, then $E$ corepresents a homology functor $E_*$ on the category of simplicial sets via $E_*(X) = \pi_*(E \spectrasmash X)$. Bousfield used this to construct a homotopy category of spaces where maps which induce isomorphisms on $E_*$--homology are isomorphisms \cite{Bou75}. Later, this was extended to a similar construction for spectra in \cite{bou79}. We recap some of the definitions from this work. We give them for simplicial sets, but there are obvious analogues for spectra. We denote homotopy classes of maps of simplicial sets by $[-,-]$ and we denote the product in $\sset$ by $\ssetsmash$.

\begin{definition}
A map $f \co X \to Y$ of simplicial sets is an \textbf{$E$--equivalence} if $E_*(f)$ is an isomorphism.
A simplicial set $Z$ is \textbf{$E$--local} if $f^* \co [Y,Z] \to [X,Z]$ is an isomorphism for
all $E$--equivalences $f \co X \to Y$. A simplicial set $A$ is \textbf{$E$--acyclic} if
$[A,Z]$ consists of only the trivial map, for all $E$--acyclic $Z$.
An $E$--equivalence from $X$ to an $E$--local object $Z$ is called an \textbf{$E$--localisation}.
\end{definition}

Bousfield localisation of simplicial sets gives rise to a homotopy theory that is particularly sensitive towards $E_*$ and $E$--local phenomena. The $E$--local homotopy theory is obtained from the category of simplicial sets by formally inverting the $E$--equivalences. In terms of model structures we have the theorem below which summarises \cite[Section 10]{Bou75}. Note that any weak homotopy equivalence of simplicial sets
is an $E$--equivalence. 

\begin{theorem}
Let $E$ be a homology theory. Then there is a model structure, $L_E\sset$, on the category of simplicial sets such that
\begin{itemize}
\item the weak equivalences are the $E_*$--isomorphisms
\item the cofibrations are cofibrations of simplicial sets (i.e. inclusions)
\item the fibrations are those maps with the right lifting property towards trivial cofibrations.
\end{itemize}
\end{theorem}

The fibrant replacement functor of this model structure is an $E$--localisation functor. In the $E$--local homotopy category of simplicial sets, $\Ho(L_E\sset)$, every object is isomorphic to a local one. Finally, we 
can identify the fibrant objects of this model structure. Since we will need to refer to this later, 
we give it as a corollary.
\begin{corollary}\label{cor:localmappingspace}
A simplicial set $K$ is $E$--fibrant if and only if it is fibrant in $\sset$ and $E$--local.
\end{corollary}
\qed

\begin{ex}
In \cite{Bou75}, Bousfield gives some examples of $E$--local simplicial sets. For this, one has to consider ``nilpotent spaces'', i.e. simplicial sets on whose homotopy groups the fundamental group acts in a certain way \cite[4.2]{Bou75}. For example, simply connected simplicial sets are nilpotent. Now let $P$ be a set of primes. For $$R=\bigoplus\limits_{p \in P}\mathbb{Z}/p \,\,\,\mbox{or}\,\,\, R=\mathbb{Z}_{(P)},$$
$\h R$--local simplicial sets can be characterised by their homotopy groups together with the action of $\pi_1$ on them \cite[Theorem 5.5]{Bou75}. In the case of $R=\mathbb{Z}_{(P)}$ this implies that 
\[
\pi_*(L_{\h R} K) \cong \pi_*(K) \otimes \mathbb{Z}_{(P)}.
\]
In the case of $R=\bigoplus\limits_{p \in P}\mathbb{Z}/p$, a simplicial set is $\h R$--local if and only if it is $P$--complete. 
\end{ex}

\bigskip
In the later sections of this paper we will deal with spectra instead of simplicial sets. Two categories of spectra will occur, most prominently the category of \textbf{sequential spectra} (or \textbf{Bousfield--Friedlander spectra}) which we will denote by $\mathcal{S}$. For some results we will need a monoidal model category of spectra. For this we choose \textbf{symmetric spectra $\mathcal{S}^\Sigma$} in the sense of \cite{HSS}. Again, there are $E$--local versions of both model categories where the weak equivalences are $E_*$--isomorphisms, cofibrations are the same as before and fibrations are defined via their lifting property. As for references, the introduction of \cite{bou79} as well as \cite[Remark 3.12]{GJ98} cover the case of $L_E \mathcal{S}$. The existence of $L_E \mathcal{S}^\Sigma$ is well-known but has not yet been fully published. The most complete reference known to the authors is the Diplom thesis of Jan M\"ollers under the supervision of Stefan Schwede.

\begin{ex}\label{ex:plocalspectra}
A spectrum $X \in \Sp$ is fibrant in the $\h\mathbb{Z}_{(P)}$--local model structure if and only if it is an $\Omega$--spectrum and its homotopy groups are $\mathbb{Z}_{(P)}$--local. In particular, this implies that $X$ is $\h\mathbb{Z}_{(P)}$--local if its level spaces are local, see 
Lemma \ref{lem:HRfamiliarity}. Unfortunately, this does not hold for $\h R$--localisation with $R=\bigoplus\limits_{p \in P}\mathbb{Z}/p$.
\end{ex}

\section{Some model category techniques and simplicial methods}\label{sec:modsimp}

In this section, we are briefly going to recall some of the definitions we work with. For more detail, we refer to \cite[Chapter 4]{Hov99} and  \cite[Appendix A]{Dug06}. 

\begin{definition}
Let $\mathcal{C}$, $\mathcal{D}$ and $\mathcal{E}$ be categories. An \textbf{adjunction of two variables} consists of functors

\[
\begin{array}{r@{\quad : \quad}l}
-\tensor - & \mathcal{C} \times \mathcal{D} \longrightarrow \mathcal{E} \\
(-)^{(-)} & \mathcal{D}^{op} \times \mathcal{E} \longrightarrow \C \\
\map(-,-) & \mathcal{C}^{op} \times \mathcal{E} \longrightarrow \mathcal{D} 
\end{array}
\]

satisfying the usual adjointness conditions, see \cite[Definition 4.1.12]{Hov99}.
\end{definition}

If the categories in above definition are model categories, then it makes sense to ask for an adjunction of two variables to be compatible with the respective model structures. 

\begin{definition}
Now let $\C$, $\mathcal{D}$ and $\mathcal{E}$ be model categories. A \textbf{Quillen adjunction of two variables} is an adjunction of two variables such that:

If $f: U \longrightarrow V$ is a cofibration in $\C$ and $g: W \longrightarrow X$ is a cofibration in $\mathcal{D}$, then the induced pushout-product map
\[
f \Box g:  (U \tensor X) \coprod_{U \tensor W} (V \tensor W) \longrightarrow V \tensor X.
\]
is a cofibration in $\mathcal{E}$. Furthermore, the map $f \Box g$ must be a trivial cofibration
if either of $f$ or $g$ is. 

The left adjoint $-\tensor -$ is sometimes called a \textbf{left Quillen bifunctor}.
\end{definition}

\begin{definition}\label{Dmodule}
Let $\mathcal{D}$ be a closed symmetric monoidal category with product $\times$ and unit $S$. A category $\mathcal{M}$ is a \textbf{closed $\mathcal{D}$--module category} if it has an adjunction of two variables
\[
(-\tensor-, (-)^{(-)}, \map(-,-)): \mathcal{M} \times \mathcal{D} \longrightarrow \mathcal{M}
\]
together with natural associativity isomorphisms
\[
(X \tensor D) \tensor E \longrightarrow X \tensor (D \times E)
\]
and natural unit isomorphisms
\[
X \tensor S \longrightarrow X.
\]
These isomorphisms have to make satisfy some standard coherence conditions. That is, the pentagonal diagram describing fourfold associativity must commute,
as must the triangle relating the 
two ways to obtain $X \tensor D$ from $X \tensor (S \times D)$.
\end{definition}

If $\mathcal{D}$ is a symmetric monoidal \emph{model} category, then one can ask for the $\mathcal{D}$--module structure on a model category $\mathcal{M}$ to be compatible with the model structures.

\newpage

\begin{definition}
Let $\mathcal{D}$ be a closed symmetric monoidal model category. A model category $\mathcal{M}$ is a \textbf{$\mathcal{D}$--model category} if it is a $\mathcal{D}$--module category in the sense of Definition \ref{Dmodule} satisfying the following. 
\begin{itemize}
\item $-\tensor -$ is a Quillen bifunctor.
\item Let $QS \longrightarrow S$ be the cofibrant replacement of the unit in $\mathcal{D}$ and let $X \in \mathcal{M}$ be cofibrant. Then 
\[
X \tensor QS \longrightarrow X \tensor S
\]
is a weak equivalence in $\mathcal{M}$
\end{itemize}
\end{definition}

We are interested in the case where $\mathcal{D}$ is the model category of pointed simplicial sets or symmetric spectra. 
\begin{definition}
A \textbf{simplicial model category} is an $\sset$--model category. A \textbf{spectral model category} is an $\Sp^\Sigma$--model category.
\end{definition}

\section{Framings}\label{sec:framings}

In this section we are going to recall some basic properties of cosimplicial and simplicial frames. Suppose one is studying a model category $\C$ that is not necessarily simplicial, one would still like to have a reasonable substitute for tensoring with simplicial sets or for mapping spaces. Framings provide such a generalisation. The idea is to take an object $A \in \C$, view it as a constant cosimplicial (or simplicial object) in $\C$ and then apply a particular cofibrant (respectively fibrant) replacement. The resulting cosimplicial or simplicial objects can then be used to define the desired tensor, cotensor and enrichment structures over $\sset$. Since various choices are involved in the process, this will not make $\C$ a simplicial model category. But it can at least ensure that the homotopy category $\Ho(\C)$ is a closed $\Ho(\sset)$--module. For more details on framings see, for example, \cite[Chapter 5]{Hov99} or \cite[Chapter 16]{Hir03}.

We note that for the statements in this section the simplicial case for $\C$ is dual to the cosimplicial case of $\C^{op}$, but we prefer to spell out the simplicial case anyway.

\bigskip
We begin with the cosimplicial case.
Let $\mathcal{C}$ be a category. By $\C^\Delta$ we denote the category of cosimplicial objects in $\C$. The standard model structure for this category is the Reedy model structure, which is described in \cite[Section 5.1]{Hov99}.
It is well--known that $\C^\Delta$ is equivalent to the category of adjunctions
\[
\sset \,\,\raisebox{-0.1\height}{$\overrightarrow{\longleftarrow}$}\,\, \C
\]
see, for example,  \cite[Proposition 3.1.5]{Hov99} or \cite[Theorem 16.4.2]{Hir03}. We denote the image of $A^\bullet \in \C^\Delta$ under this equivalence by $$(A^\bullet \tensor -, \C(A^\bullet,-)).$$ Note that
\begin{itemize}
\item $A^\bullet \tensor \Delta[n] = A^\bullet[n]$
\item $A^\bullet \tensor \partial \Delta[n] \longrightarrow A^\bullet\tensor\Delta[n]$ is the $n^{th}$ latching map of $A^\bullet$ \cite[Proposition 16.3.8]{Hir03}
\item $A^\bullet \tensor -$ preserves colimits.
\end{itemize}

Dually, the category $\C^{\Delta^{op}}$ of simplicial objects in $\C$ is equivalent to the category of adjunctions
\[
\sset^{op} \,\,\raisebox{-0.1\height}{$\overleftarrow{\longrightarrow}$}\,\, \C.
\]
We denote the image of an object $A_\bullet \in \C^{\Delta^{op}}$ by $(A_\bullet^{(-)}, \C(-,A_\bullet)).$
Note carefully that an adjunction
\[
\sset^{op} \,\,\raisebox{-0.1\height}{$\overleftarrow{\longrightarrow}$}\,\, \C
\]
is the same as an adjunction
\[
\sset \,\,\raisebox{-0.1\height}{$\overrightarrow{\longleftarrow}$}\,\, \C^{op}
\]
with the left and right adjoints interchanged. In the first convention the functor $A_\bullet^{(-)}$ is the right adjoint of $\C(-,A_\bullet)$. Again we have the following properties.
\begin{itemize}
\item $A_\bullet^{\Delta[n]} = A_\bullet[n]$
\item $A_\bullet^{\Delta[n]} \longrightarrow A_\bullet^{\partial\Delta[n]}$ is the $n^{th}$ matching map of $A_\bullet$ \cite[Proposition 16.3.8]{Hir03}
\item $A_\bullet^{(-)}$ takes limits of $\sset$ to colimits of $\C$. 
\end{itemize}
One must take care with the last property. For example, note that
a limit of $\sset$ is a colimit of $\sset^{op}$. 

\begin{definition}
If $\C$ is a model category, 
we say that an object $A^\bullet \in \C^\Delta$ is a \textbf{cosimplicial frame} if
\[
A^\bullet \tensor - : \sset \,\,\raisebox{-0.1\height}{$\overrightarrow{\longleftarrow}$}\,\, \C :\C(A^\bullet,-)
\]
is a Quillen adjunction.

An object $A_\bullet \in \C^{\Delta^{op}}$ is a \textbf{simplicial frame} if
\[
A_\bullet^{(-)}: \sset^{op} \,\,\raisebox{-0.1\height}{$\overleftarrow{\longrightarrow}$}\,\, \C : \C(-,A_\bullet)
\]
is a Quillen adjunction.
\end{definition}

Note that a Quillen adjunction $\C \,\,\raisebox{-0.1\height}{$\overrightarrow{\longleftarrow}$}\,\, \mathcal{D}$ is the same as a Quillen adjunction
$\C^{op} \,\,\raisebox{-0.1\height}{$\overleftarrow{\longrightarrow}$}\,\, \mathcal{D}^{op}$, under this identification a left Quillen functor $F: \C \longrightarrow \mathcal{D}$ becomes a right Quillen functor $F: \C^{op} \longrightarrow \mathcal{D}^{op}$ with respect to the opposite model structure \cite[Remark 1.1.7]{Hov99}.

\medskip
Simplicial and cosimplicial frames can be characterised as follows.

\begin{proposition}\label{framingcharacterisation}
 A cosimplicial object $A^\bullet \in \C^\Delta$ is a cosimplicial frame if and only if $A^\bullet$ is cofibrant and the structure maps $A^\bullet[n] \longrightarrow A^\bullet[0]$ are weak equivalences for $n\ge 0$.

 A simplicial object $A_\bullet \in \C$ is a simplicial frame if and only if $A_\bullet$ is cofibrant and the structure maps $A_\bullet[0] \longrightarrow A_\bullet[n]$ are weak equivalences for all $n \ge 0$.
\end{proposition}

The various ingredients to the proof can be found in \cite[Proposition 3.6.8, Example 5.2.4, Theorem 5.2.5, Proposition 5.4.1]{Hov99} and  \cite[Proposition 16.3.8]{Hir03}.

\begin{theorem}[Hovey]\label{existence} There exists a functor $\C \longrightarrow \C^\Delta$ such that the image $A^*$ of any cofibrant $A \in \C$ under this functor is a cosimplicial frame with $A^*[0] \cong A$.

There also exists a functor $\C \longrightarrow \C^{\Delta^{op}}$ such that the image $A_*$ of any fibrant $A \in \C$ under this functor is a simplicial frame with $A_*[0] \cong A$.
\end{theorem}

\begin{definition}
A functor $A \mapsto A^*$ together with a functor $A \mapsto A_*$ satisfying the conditions of Theorem \ref{existence} is called a \textbf{framing} of $\C$.
\end{definition}

The idea of the proof is to obtain the framing functor $(-)^*$ from a functorial factorisation in $\C^\Delta$ as a cosimplicial framing: a cosimplicial frame on $A$ can be viewed as the factorisation of a certain map into a cofibration followed by a trivial fibration. This map $l^\bullet A \longrightarrow r^\bullet A$, where $l^\bullet A$ is a cosimplicial object built from latching spaces and $r^\bullet A$ is the constant cosimplicial object \cite[Example 5.2.4]{Hov99}.  However, this factorisation has to be inductively set up to ensure that the cosimplicial frame $A^*$ has the correct object in level zero. This is \cite[Theorem 5.2.8]{Hov99}.

This also means that two framings of the same object $A \in \C$ are naturally weakly equivalent in $\C^\Delta$, see also \cite[Lemma 5.5.1]{Hov99}. Let $A^\circ$ be another cosimplicial frame of $A$. We consider the commutative square
\[
\xymatrix{ l^\bullet A \ar@{>->}[d]\ar[r]& A^\circ \ar@{->>}^\sim[d] \\
A^* \ar_\sim[r] \ar@{.>}[ur]& r^\bullet A
}
\]
Because the left vertical arrow is a cofibration and the right one a trivial fibration, there exists a lift in the diagram. Because of the 2-out-of-3 axiom this lift is also a weak equivalence. Hence every framing can be compared to the one obtained functorially.

The same is also true in the simplicial case if we view a simplicial frame $A_*$ as the factorisation of the canonical map $l_\bullet A \longrightarrow r_\bullet A$ into a cofibration that is a weak equivalence followed by a fibration in $\C^{\Delta^{op}}$.

\bigskip
Let us now look at a standard example of a framing.

\begin{ex}\label{simplicialexample} Let $\C$ be a simplicial category and $A \in \C$. Then $$A^\bullet = A \tensor \Delta[-],$$ i.e. the canonical cosimplicial object with $A^\bullet [n] = A \tensor \Delta[n]$, is a cosimplicial frame for $A$ by \cite[Proposition 16.1.3 and Proposition 16.6.4]{Hir03} and \cite[Remark 5.2.10]{Hov99}. In particular, for a simplicial set $K$
\[
A^\bullet\tensor K \cong A \ssetsmash K.
\]
Because any two framings of the same object $A \in \C$ are weakly equivalent (as shown above), 
for a cosimplicial frame $B^\bullet$ and a simplicial set $K$ we have that
\[
B^\bullet \tensor K \cong B^\bullet[0] \ssetsmash K.
\]
Dually, $A_\bullet$ with $A_\bullet[n]=A^{\Delta[n]}$ is a simplicial frame for $A$ \cite[Proposition 16.6.4]{Hir03}. Any simplicial frame $B_\bullet$  will satisfy
\[
B_\bullet^K \cong B_\bullet[0]^K.
\]
\end{ex}

Together with the framing functors $A \mapsto A^*$ and $A \mapsto A_*$ of Theorem \ref{existence} one obtains bifunctors
\[
\begin{array}{r@{\quad : \quad}ll}
- \tensor - & \C \times \sset \longrightarrow \C, & (A,K) \mapsto A^* \tensor K  \\
\map_l(-,-) & \C^{op} \times \C \longrightarrow \sset, & (A,B) \mapsto \C(A^*,B)  \\
(-)^{(-)} &  \sset^{op} \times \C \longrightarrow \C, & (A,K) \mapsto A_*^K  \\
\map_r(-,-) & \C^{op} \times \C \longrightarrow \sset, & (A,B) \mapsto \C(A,B_*). 
\end{array}
\]
Hovey shows in \cite[Theorem 5.4.9]{Hov99} that $$- \tensor - : \C \times \sset \longrightarrow \C$$ and $$(-)^{(-)}:  \sset \times \C^{op} \longrightarrow \C^{op}$$ (with the opposite model structure) have total left derived functors. However, these functors do not form a Quillen adjunction of two variables as the two right adjoints $\map_l$ and $\map_r$ do not generally agree: they only agree up to a zig-zag of weak equivalences in $\C$ \cite[Proposition 5.4.7]{Hov99}.

However, this means the right derived mapping spaces $R\map_l$ and $R\map_r$ agree. Hence we at least have an adjunction of two variables
\[
(-\tensor^L -, R(-)^{(-)}, R\map(-,-)):\,\, \Ho(\C) \times \Ho(\sset) \longrightarrow \Ho(\C).
\]

We also note that the functor $- \tensor -$ is not, in general, associative. This defect is also 
removed upon passage to the homotopy category. Hovey details the construction of a particular
associativity weak equivalence and thus comes to the following result \cite[Theorem 5.5.3]{Hov99}.

 \begin{theorem}[Hovey]\label{homotopymodule}
 The framing functor of Theorem \ref{existence} makes $\Ho(\C)$ into a closed $\Ho(\sset)$--module category.
 \end{theorem}

It is worth noting that for a simplicial model category $\C$, the $\Ho(\sset)$--module structure coming from framings agrees with the $\Ho(\sset)$--module structure derived from the simplicial structure \cite[Theorem 5.6.2]{Hov99}.

\section{\texorpdfstring{$E$}{E}--local cosimplicial frames}\label{sec:Elocalframings}

In this section we look at those framings that factor over $E$--local simplicial sets and establish the $E$--local analogues of the known results from the previous section.

The categories $\sset$ and $L_E \sset$ are identical as categories, so there is still a bijection between cosimplicial objects in a category $\C$ and adjunctions between $\C$ and $L_E \sset$ as before. However, we would like to look at those adjunctions that respect the $E$--local model structure on simplicial sets rather than the canonical one.

\begin{definition} We say that $A^\bullet \in \C^\Delta$ is an \textbf{$E$--local cosimplicial frame} if
\[
A^\bullet \tensor -: L_E\sset \,\,\raisebox{-0.1\height}{$\overrightarrow{\longleftarrow}$}\,\, \C: \C(A^\bullet,-)
\]
is a Quillen adjunction. We say that $A_\bullet \in \C^{\Delta^{op}}$ is a \textbf{$E$--local simplicial frame} if
\[
A_\bullet^{(-)}: L_E \sset^{op} \,\,\raisebox{-0.1\height}{$\overleftarrow{\longrightarrow}$}\,\, \C : \C(-,A_\bullet)
\]
is a Quillen adjunction.
\end{definition}

In particular this means that an $E$--local cosimplicial frame is a cosimplicial frame that factors over $L_E \sset$. We will use this definition later to specify for which model categories the $\Ho(\sset)$--action from Theorem \ref{homotopymodule} factors over a $\Ho(L_E \sset)$--action. Theorem \ref{Efamiliarcharacterisation} will say that this is the case if and only if all mapping spaces are $E$--local, or equivalently, if and only if every cosimplicial frame is $E$--local in the above sense. 

\bigskip

\begin{definition}
We say that a model category $\C$ is \textbf{$E$--familiar} if every cosimplicial frame $A^\bullet \in \C^\Delta$ is also an $E$--local cosimplicial frame and also if every simplicial frame $A_\bullet \in \C^{\Delta^{op}}$ is an $E$--local simplicial frame.
\end{definition}

\begin{corollary}\label{Eexistence}
Let $\C$ be an $E$--familiar model category. Then the framing functor $$\C \longrightarrow \C^\Delta$$ of Theorem \ref{existence} assigns to each cofibrant $A \in \C$ an $E$--local cosimplicial frame $A^*$ with $A^*[0] \cong A$.
\end{corollary}
\qed

Combining the framing functors $A \mapsto A^*$ and $A \mapsto A_*$ with the  
adjunctions
$$A^*: L_E \sset  \,\,\raisebox{-0.1\height}{$\overrightarrow{\longleftarrow}$}\,\, \C: \C(A^*,-) \,\,\,\,\mbox{and}\,\,\,\, A_*: L_E \sset  \,\,\raisebox{-0.1\height}{$\overleftarrow{\longrightarrow}$}\,\, \C: \C(-,A_*)$$ again gives rise to bifunctors
\[
\begin{array}{r@{\quad : \quad}ll}
- \tensor - & \C \times L_E \sset \longrightarrow \C, & (A,K) \mapsto A^* \tensor K  \\
\map_l(-,-)& \C^{op} \times \C \longrightarrow L_E \sset, & (A,B) \mapsto \C(A^*,B)  \\
(-)^{(-)}&  L_E \sset^{op} \times \C \longrightarrow \C, & (A,K) \mapsto A_*^K   \\
\map_r(-,-)& \C^{op} \times \C \longrightarrow L_E \sset, & (A,B) \mapsto \C(A,B_*). 
\end{array}
\]
It is now not difficult to establish an $E$--local analogue of the corresponding results in the previous section. First, let us work towards derived functors of the above. 

\begin{lemma}
Let $f: A^\bullet \longrightarrow B^\bullet$ be a morphism in $\C^\Delta$ and $g: K \longrightarrow L$ a morphism of simplicial sets. Consider the pushout-product
\[
f \Box g: Q :=(B^\bullet \tensor K) \coprod\limits_{A^\bullet \tensor K} (A^\bullet \tensor L) \longrightarrow B \tensor L.
\]
Then $f\Box g$ is a cofibration if both $f$ and $g$ are cofibrations. If $f$ is a trivial cofibration, then $f\Box g$ is a trivial cofibration. If $A^\bullet \in \C^\Delta$ is furthermore cofibrant and $g$ is a trivial cofibration, then $f\Box g$ is a trivial cofibration.

Dually, consider a morphism $p: A_\bullet \longrightarrow B_\bullet$ of simplicial objects in $\C$. Then the map
\[
\Hom_\Box(g,p): A_\bullet^L \longrightarrow A_\bullet^K \times_{B_\bullet^K} B_\bullet^L
\]
is a fibration if both $p$ and $g$ are. If in addition $p$ is an acyclic fibration, then so is $\Hom_\Box(g,p)$. If $B_\bullet$ is fibrant, $p$ is an acyclic fibration and $g$ is a fibration, then $\Hom_\Box(g,p)$ is an acyclic fibration.
\end{lemma}

\begin{proof}

For the case of $g$ being a cofibration and $f$ either a cofibration or trivial cofibration this is \cite[Proposition 5.4.1]{Hov99} because the cofibrations in $L_E \sset$ and $\sset$ are the same.

Now let $g: K \stackrel{\sim}{\hookrightarrow} L$ be a trivial cofibration in $L_E \sset$. Then $A^\bullet \tensor -$ and $B^\bullet \tensor -$ are left Quillen functors between $L_E \sset$ and $\C$ by assumption, as $\C$ is $E$--familiar. The rest of the proof proceeds the same way as \cite[Proposition 5.4.3]{Hov99}. Since it is the pushout of a trivial cofibration, the map
\[
B^\bullet \tensor K \longrightarrow Q
\]
is a trivial cofibration. Since $B^\bullet \tensor K \longrightarrow B^\bullet \tensor L$ is also a trivial cofibration, the cofibration $Q \hookrightarrow B^\bullet \tensor L$ must be trivial by the 2-out-of-3 axiom.

The case of $(-)^{(-)}$ follows by duality, analogously to \cite[Theorem 16.5.7]{Hir03}.
\end{proof}

For the existence of a total left derived functor it suffices to show that the functor sends trivial cofibrations between cofibrant objects to weak equivalences \cite[Proposition 8.4.4]{Hir03}. Hence we arrive at the following.

\begin{corollary}\label{Ederive}
Let $\C$ be an $E$--familiar model category. Then the functors
\[
- \tensor - : \C \times L_E \sset \longrightarrow \C
\]
and
\[
(-)^{(-)}:  L_E \sset \times \C^{op} \longrightarrow \C^{op}
\]
possess total left derived functors.
\end{corollary}
\qed

To distinguish between the derived functors of 
\[
- \tensor - : \C \times \sset \longrightarrow \C
\,\,\,\mbox{and}\,\,\,
- \tensor - : \C \times L_E\sset \longrightarrow \C
\]
we denote the latter by $\tensor^L_E$.

\bigskip
Let $\C$ be an $E$--familiar model category. Together with \cite[Theorem 5.4.9]{Hov99} we obtain
\begin{corollary}
The above derives to an adjunction of two variables
\[
(-\tensor^L_E -, R(-)^{(-)}, R\map(-,-)): \,\, \Ho(L_E \sset) \times \Ho(\C) \longrightarrow \Ho(\C).
\]
\end{corollary}
\qed

We recall that a closed module structure on a category consists of an adjunction of two variables, a unit isomorphism and an associativity isomorphism, see Definition \ref{Dmodule}. In our case, the above corollary is the first major step towards the following theorem. For this, we first need to state a lemma like \cite[Lemma 5.5.2]{Hov99}. In fact, there is nothing to prove in our case, as a cofibrant replacement functor in $\sset$ is also one in $L_E \sset$.

\begin{lemma}\label{levelzero}
Let $\C$ be $E$--familiar and $A \in \C$ cofibrant. Let $A^\bullet$ and $B^\bullet$ be cosimplicial frames for $A$. If two maps
\[
f: A^\bullet \longrightarrow B^\bullet
\]
agree on level zero, then their derived natural transformations
\[
A^\bullet \tensor^L_E K \longrightarrow B^\bullet \tensor^L_E K
\]
agree.
\end{lemma}
\qed

\begin{theorem}\label{Emodule}
The framing given in Corollary \ref{Eexistence} makes the homotopy category of any $E$--familiar model category into a $\Ho(L_E \sset)$--module. Moreover, the module action of $\Ho(\sset)$ given in Theorem \ref{homotopymodule} factors over this $\Ho(L_E \sset)$--action.
\end{theorem}

\begin{proof}
The proof follows the steps of the non-local version \cite[Theorem 5.5.3]{Hov99} but with different derived functors and derived products. Hence we are not going to spell out every detail.

Remember that in a monoidal model category with product $\tensor$, the derived product is defined via
\[
X \tensor^L Y = QX \tensor QY
\]
where $Q$ is the cofibrant replacement functor.

The first step of  \cite[Theorem 5.5.3]{Hov99} is constructing a weak equivalence in $\C$
\[
a: A \tensor (K \ssetsmash L) \longrightarrow (A \tensor K) \tensor L
\]
which is natural in $L$.
Because an $E$--local framing is in particular a framing, we can use this weak equivalence for our purposes.

In the non-local case Hovey then defines the associativity isomorphism as the composite
\begin{multline}
\tau_{AKL}: QA \tensor Q(QK \ssetsmash QL) \xrightarrow{QA \tensor q} QA \tensor (QK \ssetsmash QL) \xrightarrow{a} (QA \tensor QK) \tensor QL \nonumber\\
\xrightarrow{(q \tensor QL)^{-1}} Q(QA \tensor QK) \tensor QL \nonumber
\end{multline}
where $q: QX \longrightarrow X$ is the cofibrant replacement map, both in $\C$ and $L_E \sset$.
The model categories $\sset$ and $L_E \sset$ have the same cofibrations and trivial fibrations. Thus, we can choose the cofibrant replacement functor in $L_E \sset$ to be the same as in $\sset$. Hence we define our $E$--local associativity isomorphism to be simply $\tau$ as above.

After defining this, one needs to show that $\tau$ is also natural in $A$ and $K$. (It is easy to read from the construction in \cite[Theorem 5.5.3]{Hov99} that $\tau$ is ntaural in $L$.) Then, one further needs to prove that it satisfies the fourfold associativity and unit conditions, see Definition \ref{Dmodule}. The idea for each of these steps is the same: we write down the necessary diagrams and see that they do not necessarily commute strictly in $\C$. However, they commute in $\C$ in degree zero, so by Lemma \ref{levelzero}, they commute up to homotopy and hence in $\Ho(\C)$.

\bigskip
The claim about the $\Ho(\sset)$--action on $\Ho(\C)$ factoring over this $\Ho(L_E \sset)$--action is now easy to see. The total left derived functor of a left Quillen functor $F$ is defined via applying $F$ to the cofibrant replacement of an object. Since the cofibrant replacement functors in $\sset$ and $L_E \sset$ agree, we see immediately that the diagram
\[
\xymatrix{ \Ho(\C) \times \Ho(\sset) \ar[d]_{\id_{\Ho(\C)} \times L(\id)}\ar[r]^(.6){-\tensor^L -}& \Ho(\C) \\
\Ho(\C) \times \Ho(L_E \sset) \ar[ur]_{- \tensor^L_E -}& \\
}
\]
commutes and satisfies the necessary associativity and unit conditions.
\end{proof}

\section{\texorpdfstring{$E$}{E}--familiar model categories}\label{sec:Efamiliar}

It is not difficult to find some obvious examples of $E$--familiar model categories.

\begin{lemma}
Let $\C$ be a simplicial model category. Then $\C$ is $E$--familiar if and only if it is a $L_E \sset$--module category.
\end{lemma}

\begin{proof}
We saw at the end of Section \ref{sec:framings} that in the case of $\C$ being simplicial the bifunctors $- \tensor -, (-)^{(-)}, \map_l(-,-), \map_r(-,-)$ defined via framings
agree with the tensor, cotensor and mapping space functors of the simplicial structure. Most importantly, in the simplicial case the left mapping space functor $$\map_l(A,B)= \C(A^*,B)$$ and right mapping space functor $$\map_r(A,B)=\C(A,B_*)$$ agree. Hence Corollary \ref{Ederive} provides a $L_E \sset$--model category structure if and only if $\C$ is $E$--familiar.
\end{proof}

\begin{corollary}
The model category of $E$--local simplicial sets ($L_E \sset$) and 
the model category of $E$--local spectra ($L_E \mathcal{S}$) are $E$--familiar.
\end{corollary}
\qed

The following shows that the notion of an $E$--familiar model category indeed generalises the concept of $L_E \sset$--model categories.

\begin{proposition}\label{Esimplicial}
If $\C$ is $E$--familiar and simplicial, then the $\Ho(L_E \sset)$--module structure from Theorem \ref{Emodule} agrees with the $\Ho(L_E \sset)$--module structure derived from the $L_E \sset$--model category structure.
\end{proposition}

\begin{proof}
We are going to show that the identity
\[
\id: \Ho(\C) \longrightarrow \Ho(\C)
\]
is a $\Ho(L_E \sset)$--module functor. Here, the domain $\Ho(\C)$ has the $\Ho(L_E \sset)$--action given by the derived $L_E\sset$--model category structure. We give the target $\Ho(\C)$ the $\Ho(L_E \sset)$--module structure coming from framings. To show that the identity is a $\Ho(L_E \sset)$--module functor we need a natural isomorphism
\[
A \tensor^L_E K \longrightarrow A^* \tensor^L_E K
\]
satisfying two coherence diagrams \cite[Definition 4.1.7]{Hov99}. (Again, the first $\tensor^L_E$ is part of the $L_E \sset$--model category structure while the second one is coming from framings.)

Now let $A \in \C$.
We remember from Example \ref{simplicialexample} that $$A \tensor \Delta[-]$$ is an $E$--local framing on $A$, and that 
\[
(A \tensor \Delta[-]) \tensor^L_E K \cong A \tensor^L_E K.
\]
Hence by Section \ref{sec:framings} and \cite[Lemma 5.5.1]{Hov99}, there is an isomorphism in $\Ho(\C)$
\[
\sigma: A \tensor^L_E K \cong (A \tensor \Delta[-]) \tensor^L_E K \longrightarrow A^* \tensor^L_E K
\]
which is natural both in $A$ and $K$.

The first of the two coherence diagrams contains the two actions of the unit and is obvious since $A^* \tensor \Delta[0] \cong A$. Consider the second diagram:
\[
\xymatrix{ ((A \tensor \Delta[-]) \tensor^L_E K) \tensor^L_E L \ar[r] \ar[d] & (A^* \tensor^L_E K ) \tensor^L_E L \ar[dd] \\
(A \tensor \Delta[-]) \tensor^L_E (K \tensor^L_E L) \ar[d] & \\
A^* \tensor^L_E (K \tensor^L_E L) \ar[r] & (A^* \tensor^L_E K)^* \tensor^L_E L
}
\]
The upper left corner agrees with the framing
\[
\big( (A \tensor \Delta[-]) \tensor^L_E K \big) \tensor \Delta[-] \in \C^\Delta
\]
evaluated on $L$, so both clockwise and counterclockwise composition are maps of cosimplicial frames that obviously agree in degree 0.
So by Lemma \ref{levelzero}, the above diagram commutes in $\Ho(\C)$, which is what we wanted to prove. 
\end{proof}

We now provide an important characterisation of $E$--familiarity.

\begin{theorem}\label{Efamiliarcharacterisation} The following are equivalent.
\begin{enumerate}
\item The model category $\C$ is $E$--familiar.
\item The $\Ho(\sset)$--module structure on $\Ho(\C)$ factors over a $\Ho(L_E \sset)$--module structure.
\item The mapping spaces $R\map(-,-)$ are $E$--local.
\end{enumerate}
\end{theorem}

\begin{proof}

We first show the equivalence of (1) and (2). 
One direction is precisely Theorem \ref{Emodule}. As for the converse, remember that $\C$ is $E$--familiar by definition if every framing is also an $E$--local framing. This means that for every cosimplicial frame $A^\bullet$, the functor $A^\bullet \tensor - $ sends $E_*$--isomorphisms in simplicial sets to weak equivalences in $\C$. But this is exactly the case if we ask for the $\Ho(\sset)$--module structure to factor over $\Ho(L_E \sset)$.

\bigskip
Now we turn to the equivalence of (2) and (3).
One direction is straightforward: if $\C$ is $E$--familiar, then 
\[
\C(X^\bullet,-): \C \longrightarrow L_E \sset
\]
is a right Quillen functor for a cosimplicial frame $X^\bullet$. Hence it sends fibrant objects to fibrant objects. Since $E$--fibrant simplicial sets are automatically local, $\C(X^\bullet,Y)$ and hence $R\map(X,Y)$ are $E$--local.

Now let us look at the converse. We have to show that $\C(X^\bullet,-)$ sends fibrations in $\C$ to $E$--fibrations of simplicial sets. By \cite[Corollary A.2]{Dug01} it satisfies to show that $\C(X^\bullet,-)$ sends fibrations between fibrant objects to $E$--fibrations. Since 
\[
\C(X^\bullet,-): \C \longrightarrow \sset
\]
is a right Quillen functor, it sends fibrant objects in $\C$ to fibrant objects in $\sset$. By assumption, $\C(X^\bullet,Y)$ is also $E$--local for fibrant $Y$. Hence by Corollary \ref{cor:localmappingspace}, $\C(X^\bullet, Y)$ is $E$--fibrant. Since $\sset$--fibrations between $E$--fibrant objects are $E$--fibrations (see for example the proof of Proposition 3.2 in \cite{Roi07}), $$\C(X^\bullet,-): \C \longrightarrow L_E \sset$$ preserves fibrations between fibrant objects.
\end{proof}

We have to note that $E$--familiarity is certainly not an invariant of the homotopy category of a model category alone. For example, take the $K$--local stable homotopy category $\Ho(L_1 \mathcal{S})$ localised at an odd prime. (By $K$, we mean complex topological $K$--theory.) By \cite{Fra96} this possesses at least one ``exotic model''. This means that this homotopy category can be realised by at least one model category which is not Quillen equivalent to $K_{(p)}$--local spectra. It was noted in \cite{Roi07} that every framing on such an algebraic model will be trivial, whereas the framings on $L_1 \mathcal{S}$ are clearly nontrivial. Indeed, \cite{Roi07} shows that an exotic model can be detected entirely by the action of the generator 
\[
\alpha_1 \in \pi_{2p-3}^{st}(L_1 \mathbb{S}) \cong \mathbb{Z}/p
\]
via framings. We will investigate this in more detail in Section \ref{sec:applications}.

\section{Stable frames}\label{sec:Stableframes}

It is worthwhile to ask whether stable model categories provide framings with more interesting and useful structure. One natural task would be investigating the possibility of replacing simplicial sets, $\sset$,  by sequential spectra, $\mathcal{S}$, in all of the previous sections if $\C$ is stable. A first step towards this was undertaken by Schwede and Shipley \cite{SchShi02} where they show the ``Universal Property of Spectra''.

\begin{theorem}[Schwede--Shipley]\label{thm:ssadjunct}
Let $\C$ be a stable model category and $X$ a fibrant and cofibrant object of $\C$. Then there is a Quillen adjunction
\[
X \sstensor -: \mathcal{S} \lradjunction \C: \Map(X,-)
\]
such that $X \sstensor \mathbb{S} \cong X$. 
\end{theorem}

Fabian Lenhardt later generalised this to the context of stable framings in \cite{Len11}. He specifies the category of adjunctions $$ \mathcal{S} \lradjunction \C$$ and characterises those which give rise to Quillen adjunctions, giving a notion of stable (cosimplicial) frames. He then proceeds to show that each cofibrant--fibrant object in $\C$ possesses such a stable frame. Finally he describes how for stable $\C$, these constructions equip $\Ho(\C)$ with the structure of a closed $\Ho(\mathcal{S})$--module category. In order to $E$--localise these results, let us give the most important definitions and results of \cite{Len11} first. 

For this, it is not always necessary to assume $\C$ to be stable, but we are going to do so for the rest of this section for convenience. 

\bigskip
We remember that the category of adjunctions $$\sset \lradjunction 
C$$ is equivalent to cosimplicial objects $\C^\Delta$. We are now going to describe the category that is equivalent to adjunctions $$ \mathcal{S} \lradjunction \C.$$ 
First of all, let $X \in \C^\Delta$ be a cosimplicial frame. We are going to define the suspension $\Sigma X$ of $X$ as the cosimplicial object corresponding to the adjunction $X \sframe ( - \ssetsmash S^1)$.

\begin{definition}
A \textbf{$\Sigma$--cospectrum} is a sequence of objects $X_n \in \C^\Delta$ together with structure maps
\[
\Sigma X_n \longrightarrow X_{n-1}.
\]
A morphism of $\Sigma$--cospectra consists of a sequence of morphisms in $\C^\Delta$ that are compatible with the structure maps. The resulting category is denoted $\C^\Delta(\Sigma)$.
\end{definition}

Furthermore, $\C^\Delta(\Sigma)$ can be equipped with a useful model structure, 
see \cite[theorem 3.11]{Len11}. The following result is Theorem 3.7 of 
that paper. 

\begin{theorem}[Lenhardt]
The category $\C^\Delta(\Sigma)$ is equivalent to the category of adjunctions $$\mathcal{S} \lradjunction \C.$$ 
The image of a cospectrum $X$ under this equivalence is denoted by $$(X \sframe  -, \Map(X,-)).$$ 
\end{theorem}

The key to this is the following idea. Precomposing an adjunction
\[
L: \mathcal{S} \lradjunction \C: R
\]
with the adjunctions
\[
F_n: \sset \lradjunction \mathcal{S}: \ev_n
\]
(see \cite[Definition 2.1]{Len11}) for $n \ge 0$ gives a sequence of adjunctions
\[
L_n: \sset \lradjunction \C :R_n.
\]
Each of these is characterised by a cosimplicial object $X_n \in \C^\Delta$. These give the ``level spaces'' of a cospectrum $X \in \C^\Delta(\Sigma)$. 

Further, there are natural transformations
\[
\tau_n: L_n \circ \Sigma \longrightarrow L_{n-1}
\]
and their adjoints
\[
\eta_n: R_{n-1} \longrightarrow \Omega \circ R_n,
\]
see \cite[Proposition 3.4]{Len11}. These give rise to morphisms of cosimplicial sets $$\Sigma X_n \longrightarrow X_{n-1},$$ which are the structure maps of the cospectrum $X$. 

Lenhardt's Proposition 3.4 says that an adjunction $(L,R)$ as above is uniquely determined by either the $L_n$ and $\tau_n$ or the $R_n$ and $\eta_n$, which proves his Theorem 3.7 as quoted above.

\medskip
He then proceeds by characterising those cospectra that give rise to Quillen adjunctions in \cite[Section 6]{Len11}. 

\begin{proposition}[Lenhardt]
The adjunction
\[
X \sframe -: \mathcal{S} \lradjunction \C: \Map(X,-)
\]
is a Quillen adjunction if and only if 
\begin{itemize}
\item each $X_n$ is a cosimplicial frame
\item the structure maps $\Sigma X_n \longrightarrow X_{n-1}$ are weak equivalences.
\end{itemize}
Such a cospectrum $X$ is called a \textbf{stable frame}.
\end{proposition}

Furthermore, each object in $\C$ possesses a framing \cite[Theorem 6.3]{Len11}:

\begin{theorem}
Let $A \in \C$ be a fibrant and cofibrant. Then there is a stable frame $X$ with $X_{0,0} = X \sframe \mathbb{S} \cong A$. 
\end{theorem}

In particular this implies Schwede's and Shipley's Universal Property of Spectra.

\bigskip
Unfortunately, stable frames cannot be chosen with such good functorial properties as their unstable analogues, 
as is noted by \cite[Remark 6.4]{Len11}. The problem is of a categorical nature and arises whenever $\C$ is not a simplicial model category.  While the suspension functor $\Sigma$ of $\Ho(\C)$ can be realised via the use of framings $S^1 \tensor -$, the adjoint of this functor is unlikely to be $\Omega$. This seems to seems to prevent one from being able make a functorial construction of stable frames. 

\bigskip
The central structural result is a stable version of Theorem \ref{homotopymodule}, it appears as 
\cite[Theorem 7.3]{Len11}. 

\begin{theorem}[Lenhardt]\label{stablemodule}
Let $\C$ be a stable model category. Via stable frames, $\Ho(\C)$ becomes a closed $\Ho(\mathcal{S})$--module category. 
\end{theorem}

Just as framings in $\sset$ provide a generalisation of a simplicial model category structure, a stable framing does the analogue for spectral model categories. By ``spectral model category'' we mean a $\mathcal{S}^\Sigma$--model category, where $\mathcal{S}^{\Sigma}$ denotes symmetric spectra.  Symmetric spectra are Quillen equivalent to sequential spectra via the Quillen equivalence
\[
V: \mathcal{S} \lradjunction \mathcal{S}^\Sigma: U
\]
where the right adjoint $U$ is forgetting the symmetric action, see \cite[Proposition 4.2.4]{HSS}. 
As with sequential spectra, there is a free spectrum and evaluation adjunction $(F_n^\Sigma,\ev_n)$ between simplicial sets and symmetric spectra, see \cite[Definition 2.1.7]{HSS}. It factors over the non-symmetric case as
\[
\xymatrix@C+0.3cm@R+0.3cm{ \sset \ar@<0.7ex>[rr]^{F_n^\Sigma} \ar@<0.7ex>[rd]^{F_n} & & \mathcal{S}^\Sigma \ar[dl]^{U} \ar[ll]^{\ev_n} \\
& \ar@<0.7ex>[ur]^{V} \ar[ul]^{\ev_n} \mathcal{S}.& 
}
\]
With this we can write down what framings in spectral model categories look like and observe that framings are indeed a generalisation of the spectral structure.

\begin{ex}\label{stableexample}
If $\C$ is a spectral model category and $X \in \C$ is fibrant and cofibrant, then we have a Quillen adjunction
\[
X \sframe -: \mathcal{S}^\Sigma \lradjunction \C: \Map(X,-)
\]
which is part of the spectral structure. Precomposing with the adjunction $(V,U)$ as described above gives an adjunction
\[
\mathcal{S} \lradjunction \mathcal{S}^\Sigma \lradjunction \mathcal{C}
\]
which we are also going to denote by $(X \sframe -, \Map(X,-)).$ We can now easily describe the corresponding cospectrum $\overline{X}$. Its $n^{th}$ level, $X_n \in \C^\Delta$, is the cosimplicial set corresponding to the adjunction 
$$X \sframe F_n(-): \sset \lradjunction \C :
\ev_n \circ \Map(X,-)
$$ 

and the structure maps $$\Sigma X_n \longrightarrow X_{n-1}$$ are obtained via applying the functor $X \sframe -$ to the natural transformation
$$F_n \circ \Sigma \longrightarrow F_{n-1}.$$ This natural transformation induces the trivial map in level $n-1$ and below,  and the identity in level $n$ and above. When  evaluated on a simplicial set $K$, it gives a weak equivalence of sequential spectra. Hence the structure maps $$\Sigma X_n =X \sframe F_n \circ \Sigma \longrightarrow X \sframe F_{n-1} = X_{n-1}$$ are weak equivalences of cosimplicial objects in $\C$, as required. 

Thus the cospectrum $\overline{X}$ defines a stable frame with $X_{0,0} = X$. By uniqueness of stable frames \cite[Proposition 4.7]{Len11}, every stable frame $Y$ on an object $X \in \C$ will agree, up to homotopy, with the Quillen pair $(X \sframe -, \Map(X,-))$ given by the spectral structure. 
\end{ex}

We can put this example in a context with even higher structure, the following result appears as 
\cite[Theorem 7.4]{Len11}.  

\begin{theorem}\label{stablemodulestructure}
Let $\C$ be a spectral model category. Then the $\Ho(\mathcal{S})$ module structure derived from the spectral structure agrees with the $\Ho(\mathcal{S})$--module structure coming from framings as in Theorem \ref{stablemodule}.
\end{theorem}

For this, we remember that although the category of sequential spectra $\mathcal{S}$ is not a monoidal model category, the stable homotopy category $\Ho(\mathcal{S})$ is monoidal. Further, $\mathcal{S}$ and symmetric spectra $\mathcal{S}^\Sigma$ are Quillen equivalent, hence $\Ho(\mathcal{S}^\Sigma)=\Ho(\mathcal{S})$. 
This result is also a special case (the one where $E_* = \pi_*$) of Proposition \ref{stableagree}.

\section{\texorpdfstring{$E$}{E}--familiarity and stable model categories}\label{sec:Efamiliarstable}

We are now interested in $E$--local versions of those results. The central application we have in mind is obtaining a ``Universal Property of $E$--local spectra''
analogous to Theorem \ref{thm:ssadjunct}.

\begin{definition}
We say that a $\Sigma$--cospectrum $X$ is an \textbf{$E$--local stable frame} if
\[
X \sframe -: L_E \mathcal{S} \lradjunction \C: \Map(X,-)
\]
is a Quillen adjunction.
We further say that the model category $\C$ is \textbf{stably $E$--familiar} if every stable frame is also an $E$--local stable frame. 
\end{definition}

Let us first make some immediate observations following this definition, 
\begin{lemma}
Any $L_E \mathcal{S}^\Sigma$--model category is stably $E$--familiar. 
\end{lemma}

\begin{proof}
This follows from Example \ref{stableexample} in combination with Theorem \ref{localmappingspectrum}.
\end{proof}

\begin{lemma}\label{lem:evaluation}
Any stably $E$--familiar model category is also $E$--familiar.
\end{lemma}

\begin{proof}
We must show that 
for pair of any objects $X$ and $Y$ in $\C$ the simplicial set 
$R \map(X,Y)$ is $E$--local. 
For $Z$ a cospectrum there is an equality of functors
\[
\ev_0(\Map(Z,-))=\map(Z,-)_0=\map(Z_0,-) : \C \longrightarrow \sset.
\]
Hence, on homotopy categories, there is an isomorphism of functors
\[
R \ev_0 \circ R \Map(-,-) \cong R \map(-,-) : \Ho(\C)^{op} \times \Ho(\C) \longrightarrow \Ho(\sset).
\]
Since $R \Map(-,-)$ takes values in $\Ho(L_E \Sp)$ and $R \ev_0$ can also be thought of as a functor from
$\Ho(L_E \Sp)$ to $\Ho(L_E \sset)$, it follows that 
for any $X$ and $Y$ in $\C$, $R \map(X,Y)$ must be an $E$--local simplicial set.

\end{proof}

By $\omega X$ we denote any stable frame on $X$. (This is consistent with Lenhardt's notation.)
We also note that the bifunctor
\[
\C \times L_E \Sp \longrightarrow \C, \,\,\, (X,A) \mapsto \omega X \sframe A
\]
possesses a total left derived functor. Since this is very similar to \cite[Corollary 6.6]{Len11} and our previous work in Section \ref{sec:Elocalframings}, we omit the proof. We denote this derived functor by $\sframe^L_E$. 

\begin{lemma}\label{Ederivedfunctor}
Let $\C$ be a simplicial and stably $E$--familiar model category. Further, let 
\[
F,G : X \longrightarrow Y
\]
be two maps of stable frames $X$ and $Y$ on $\C$ that agree on the sphere $\mathbb{S}$. Then the derived natural transformations
\[
X \Esframe - \longrightarrow Y \Esframe -
\]
induced by $F$ and $G$ agree. 
\end{lemma}

Again, this requires no proof, we simply note that this uses \cite[Corollary 4.11]{Len11}. We see that a cofibrant replacement functor in $\Sp$ is automatically a cofibrant replacement functor in $L_E \Sp$. 

Now that we have established some of the properties that a stably $E$--familiar model category possesses, we can turn to the stable analogues of Theorem \ref{Emodule}, Proposition \ref{Esimplicial} and Theorem \ref{Efamiliarcharacterisation}.

\begin{theorem}\label{thm:EfamiliarEmodule}
Let $\C$ be stably $E$--familiar, then $\Ho(\C)$ is a $\Ho(L_E \mathcal{S})$--module category. 
Moreover, a stable model category $\C$ is stably $E$--familiar if and only if the $\Ho(\mathcal{S})$--module structure given by Theorem \ref{stablemodulestructure} factors over this module structure. 
\end{theorem}

\begin{proof}
We need to construct an associativity isomorphism
\[
X \Esframe (K \Esmash L) \longrightarrow (X \Esframe K) \Esmash L
\]
that is natural in $X \in \C$ and $K, L \in \mathcal{S}^\Sigma$ and satisfies various coherence conditions, see our previous work in Theorem \ref{Emodule}. We begin with $X \in \C$ being fibrant and cofibrant. By $\overline{K}$ we denote the stable frame construction for a spectral category introduced in Example \ref{stableexample}.

Now consider the stable frames
\[
\omega X \sframe ( \overline{K} \sframe -) \,\,\mbox{and}\,\, \omega(\omega X \sframe K) \sframe -.
\]
Note that the first functor is a stable frame via composition of Quillen functors. They are both stable frames on the object $\omega X \sframe K \in \C$, so by \cite[Theorem 6.10]{Len11} we get a weak equivalence, natural in $L$,
\[
a: \omega X \sframe ( {K} \spectrasmash L) \longrightarrow \omega(\omega X \sframe K) \sframe L,
\]
remembering that $\overline{K} \sframe L = K \spectrasmash L$.
As in \cite[Theorem 5.5.3]{Hov99} we define our associativity isomorphism as the composite
\begin{multline}
\tau: \omega QX \sframe Q( {QK} \spectrasmash QL) \overset{1 \sframe q}{\to} 
\omega QX \sframe ( {QK} \spectrasmash QL)     \overset{a}{\to}
\omega(\omega QX \sframe QK) \sframe QL)  \\ \xrightarrow{(q \sframe 1)^{-1}}
\omega Q(\omega QX \sframe QK) \sframe QL)
\nonumber
\end{multline}
To show the necessary naturality and coherence conditions, we employ the same strategy as in previous proofs: we write down diagrams in $\C$ that do not necessarily commute. But since they commute in bidegree $(0,0)$, we can use \cite[Theorem 6.10 (b)]{Len11} and deduce that they commute in $\Ho(\C)$, which is what we are really after. 

The first diagram shows naturality in $X$. Let $X \longrightarrow Y$ be a morphism between fibrant and cofibrant objects in $\C$, then we have the diagram below, which will not usually commute. 
\[
\xymatrix{ \omega X \sframe (\overline{K} \sframe -) \ar[r]\ar[d] & \omega( \omega X \sframe K) \sframe - \ar[d] \\
\omega Y \sframe (\overline{K} \sframe -) \ar[r]& \omega( \omega Y \sframe K) \sframe - 
}
\]
Both clockwise and counterclockwise composites agree on the sphere spectrum $\mathbb{S}$, so by \cite[Theorem 6.10 and Corollary 6.11]{Len11} the above diagram commutes in $\Ho(\C)$, which we wanted to show. Naturality in $K$ is proved in a very similar fashion, so we omit it. 

Next, we prove fourfold associativity similarly to \cite[Theorem 5.5.3]{Hov99} using 
\cite[Corollary 6.11]{Len11}. The fourfold associativity diagram is
\[
\xymatrix{ 
\omega X \Esframe (\overline{K} \sframe (\overline{L} \sframe -)) 
\ar[r]^{(4)} \ar[d]_{(1)} 
& \omega(\omega X \Esframe K) \Esframe (\overline{L} \sframe -) 
\ar[dd]^{(5)} \\
\omega X \Esframe ((\overline{K \spectrasmash L}) \sframe -) 
\ar[d]_{(2)}  \\
\omega( \omega X \Esframe (\overline{K \sframe L})) \Esframe {-} \ar[r]_{(3)} 
& \omega( \omega ( \omega X \Esframe K) \Esframe L ) \Esframe -
}
\]
The map (1) is the identity on $\omega X$ applied to the associativity isomorphism in $\Ho(\mathcal{S})$. (We note that we discussed in Section 6 how the framing action agrees with the action derived from the spectral model structure.) The map (2) is any map covering the identity of $\omega X \Esframe (K \spectrasmash L)$ and the map (3) is $\omega \tau \Esframe -$, that is, any map of framings covering $\tau$. 

Now we turn to the clockwise maps, evaluated on the sphere $\mathbb{S}$, (4) is just $\tau$ and  
(5) is any map covering the identity on $(\omega X \Esframe K) \Esframe L$. 

If we evaluate each on the sphere then both the clockwise and anticlockwise composites are
just applications of $\tau$. Hence on homotopy categories these two composite maps agree. 
It follows that, as natural transformations of functors on homotopy categories, 
the diagram commutes, which is precisely the statement that four-fold associativity is coherent. 
\end{proof}

We can also ask if a Quillen adjunction between stably $E$--familiar model categories
is compatible with the $\Ho(L_E \mathcal{S})$--actions on the homotopy categories. 
the answer is Lemma \ref{lem:localmodfunctors}, which shows that any Quillen pair will be compatible. 

The next theorem establishes that $E$--local stable frames are indeed a generalisation of $L_E \mathcal{S}^\Sigma$--model category structures.

\begin{proposition}\label{stableagree}
Let $\C$ be a $L_E \mathcal{S}^\Sigma$--model category. Then the $\Ho(L_E \mathcal{S})$--module structure on $\Ho(\C)$ induced by stable framings agrees with the $\Ho(L_E \mathcal{S})$--module structure given by the $L_E \mathcal{S}^\Sigma$--model category structure.
\end{proposition}

\begin{proof}
We show that the identity
\[
\id: \Ho(\C) \longrightarrow \Ho(\C) 
\]
is a $\Ho(L_E \mathcal{S})$--module functor, similar to what we did in Proposition \ref{Esimplicial}. Here, the domain has the $\Ho(L_E \mathcal{S})$--action that is derived from the $L_E \mathcal{S}^\Sigma$ model category structure. The module structure on the codomain is induced by $E$--local stable frames. 

This means we have to construct a natural isomorphism
\[
X \Esmash K \longrightarrow \omega X \Esframe K
\] 
where the first product is part of the $L_E \mathcal{S}^\Sigma$--structure and $\omega X$ is a stable framing for $X \in \C$.

We saw in Example \ref{stableexample} that there is a framing $\overline{X}$ on the object $X$ using the spectral structure that agrees with $X \Esmash -$. So, by \cite[Proposition 4.7]{Len11}, there is a map extending the identity on level (0,0) to a map of stable framings. By Lemma \ref{Ederivedfunctor}, this induces the desired isomorphism above. 

We have to show that it satisfies the necessary coherence conditions. Again, the unit condition is easily seen. Now we consider the diagram
\[
\xymatrix{ (\overline{X} \sframe K) \spectrasmash L \ar[d]\ar[r]& \omega (X \spectrasmash K) \sframe L \ar[dd] \\
\ar[d] \overline{X} \sframe (K \spectrasmash L) & \\
\ar[r]  \omega X \sframe (K \spectrasmash L) & \omega( \omega X \sframe K) \sframe L 
}
\]
The functor $(\overline{X} \sframe K) \sframe -$ agrees with the functor $\overline{(\overline{X} \sframe K)} \sframe -$, which is a stable frame for the object $X \spectrasmash K \in \C$. But so is $ \omega( \omega X \sframe K) \sframe -$.  Together with Lemma \ref{Ederivedfunctor} we hence see that the clockwise and counterclockwise compositions in the above diagram commute in $\Ho(\C)$, which is what we wanted to prove.

\end{proof}

Recall that a localisation functor $L_E$ is \textbf{smashing} if the map $$X \longrightarrow X \spectrasmash L_E \mathbb{S}$$ is an $E$--localisation for any spectrum $X$. Examples of smashing localisations include the Johnson--Wilson theories $E(n)$, which we are going to talk about in more detail in Section \ref{sec:applications}. A example of a localisation that is not smashing is localising with respect to a Morava $K$--theory, $K(n)$. In the case of a smashing localisation there is a relatively simple criterion for being stably $E$--familiar. 

\begin{proposition}
Let $E$ be a homology theory for which $L_E$ is smashing. Then $\C$ is stably $E$--familiar if and only if  the map
\[
X \sframe \lambda: X \cong X \sframe \mathbb{S} \longrightarrow X \wedge L_E \mathbb{S}
\]
is a weak equivalence in $\C$ for all stable frames $X$.

If furthermore $\C$ has a set of small weak generators $\mathcal{G}$, then $\C$ is
stably $E$--familiar if and only if the map 
$Y \to Y \sframe^L L_E \mathbb{S}$
is a weak equivalence for each $Y \in \mathcal{G}$.
\end{proposition}

\begin{proof}
The ``only if'' part is obvious: the map $\lambda: \mathbb{S} \longrightarrow L_E \mathbb{S}$ is an $E$--equivalence. So if $\C$ is stably $E$--familiar, $X \sframe \lambda$ is a weak equivalence in $\C$ by definition.

Conversely, assume that $X \wedge \lambda$ is a weak equivalence. To show that $\C$ is stably $E$--familiar we need to show that the functor $X \sframe -$ sends $E$--equivalences to weak equivalences in $\C$. Let $f: K \longrightarrow L$ be an $E$--equivalence of spectra. Then the following diagram commutes.
\[
\xymatrix{ X \sframe K  \ar[r]^{X \sframe f}\ar[d]_{\sim} & X \sframe L \ar[d]^{\sim}\\
X \sframe (L_E \mathbb{S} \spectrasmash K) \ar[r] & X \sframe (L_E \mathbb{S} \spectrasmash L) 
}
\]
By assumption, the vertical maps are weak equivalences in $\C$. Since $E$ is smashing, the spectra $L_E \mathbb{S} \spectrasmash K$ and $L_E \mathbb{S} \spectrasmash L$ are $E$--local. The map $f$ is an $E$--equivalence, so it also induces an $E$--equivalence between $L_E \mathbb{S} \spectrasmash K$ and $L_E \mathbb{S} \spectrasmash L$. But $E$--equivalences between $E$--local spectra are $\pi_*$--isomorphisms. We know that $X \sframe -$ sends $\pi_*$--isomorphisms to weak equivalences in $\C$, so the bottom horizontal arrow in the above diagram is also a weak equivalence. By the 2-out-of-3 axiom the top horizontal arrow is also a weak equivalence, as required. 

The second statement follows since any element of $\Ho(\C)$ can be built from 
the generators via coproducts and triangles, which are preserved by 
$\sframe^L$.
\end{proof}

We now state
the central characterisation of stable $E$--familiarity.

\begin{theorem}\label{localmappingspectrum}
A model category $\C$ is stably $E$--familiar if and only if every homotopy mapping spectrum $\Map(X,Y)$ is an $E$--local spectrum.
\end{theorem}

\begin{proof}
The ``only if'' part is simple: $\Map(X,Y)$ sends sends fibrant objects to $E$--fibrant spectra, and those are local.

As for the converse, assume that $\Map(X,Y)$ is $E$--local for fibrant $Y$. The functor
\[
\Map(X,-): \C \longrightarrow \mathcal{S}
\]
preserves trivial fibrations, so
\[
\Map(X,-): \C \longrightarrow L_E \mathcal{S}
\]
also does. Thus we still need to show that $\Map(X,-)$ preserves fibrations. This is done in the following four steps, similar to \cite[Proposition 3.2]{Roi07}. 
\begin{enumerate}
\item The functor $\Map(X,-)$ preserves fibrant objects.
\item The functor $\Map(X,-)$ sends fibrations to level fibrations.
\item In $L_E \mathcal{S}$, level fibrations between fibrant objects are fibrations.
\item If a functor that preserves trivial fibrations also preserves fibrations between fibrant objects, it is a right Quillen functor.
\end{enumerate}
The fibrant objects of $L_E \mathcal{S}$ are the $E$--local $\Omega$--spectra. For fibrant $Y$, the spectrum $\Map(X,Y)$ is an $\Omega$--spectrum by construction. Further, it has been assumed to be $E$--local, so (1) is satisfied. 
The second point is again satisfied by construction as $$\Map(X,Y)_n=\map(X_n,Y)$$ with $X_n$ a cosimplicial frame.
The third point has been proved explicitly in \cite[Proposition 3.2]{Roi07}.
Finally, (4) is Corollary A.2 in \cite{Dug01}. This completes the proof. 
\end{proof}

Composition of morphisms in $\C$ makes $R\Map(X,Y)$ into a module spectrum over $R\Map(X,X)$. (Here, we mean ring and module objects in the stable homotopy category rather than referring to structured ring spectra in the underlying model categories.) Since module spectra over $E$--local spectra are again $E$--local, provided $E$ is a ring spectrum,  \cite[Proposition 1.17]{Rav84}, we can also state the following. 

\begin{corollary}\label{cor:localringspectrum}
If $E$ is a ring spectrum, then model category $\C$ is stably $E$--familiar if and only if the spectra $R\Map(X,X)$ are $E$--local for all $X \in \C$.
\end{corollary}
\qed

Note that if $L_E$ is smashing, then $L_E \mathbb{S}$ is a ring spectrum and $L_E = L_{L_E \mathbb{S}}$, so 
the above holds for all smashing localisations. 

For the special case $E= K_{(2)}$, this criterion was the key point in the main result of \cite{Roi07}. We are going to investigate this relation further in Subsection \ref{rigidity}.

\bigskip
We can also conclude that being stably $E$--familiar is invariant under Quillen equivalence. 

\begin{lemma}\label{lem:Quilleninvariance}
Let $F: \C \lradjunction \mathcal{D}: G$ be a Quillen equivalence. Then $\C$ is stably $E$--familiar if and only if $\mathcal{D}$ is. 
\end{lemma}

\begin{proof}
The heart of this proposition is Theorem \ref{localmappingspectrum}, 
along with the fact that if there is a $\pi_*$--isomorphism of spectra 
$f \co X \to Y$ then $X$ is $E$--local if and only if $Y$ is. 
Thus we must show that the mapping spectra of these two categories agree. The key input 
to this is \cite[Theorem 7.3]{Len11} which states that the functor $LF$ is a 
$\Ho(\Sp)$--module functor. 

Take $C \in \mathcal{C}$ and $D \in \mathcal{D}$,
then by a standard adjunction argument the spectra $R \Map(LF(C), D)$ and $R \Map(C, RG(D))$ are weakly equivalent.

Now we have all the pieces ready. If $\mathcal{D}$ is stably $E$--familiar, 
then take any pair of objects $C_1$, $C_2$ in $\mathcal{C}$. 
Since we have a Quillen equivalence, the unit of the derived adjunction,
$\id \to RGLF$,
induces a weak equivalence of spectra 
$$R \Map (C_1, C_2) \to 
R \Map (C_1, RGLF(C_2)) \simeq 
R \Map (LF(C_1), LF(C_2)).$$ 
The right hand side of the above is 
which is $E$--local as $\mathcal{D}$ is stably
$E$--familiar. Thus all mapping spectra of $\mathcal{C}$ are $E$--local.

Conversely, assume that $\mathcal{C}$ is stably $E$--familiar,
then for any $D_1$ and $D_2$ of $\mathcal{D}$
the mapping spectrum $R \Map (D_1, D_2)$ is weakly equivalent to 
$R \Map (LF RG(D_1), D_2)$ as the counit of the derived adjunction
is a weak equivalence. By adjunction as before we can conclude that 
$R \Map (D_1, D_2)$ is stably equivalent to 
the $E$--local spectrum $R \Map (RG(D_1), RG(D_2))$. Thus all 
mapping spectra of $\mathcal{D}$ are $E$--local. 
\end{proof}

\bigskip
Unfortunately, ``stable'' together with ``$E$--familiar'' does not imply ``stably $E$--familiar''. We are going to look at the difference in the next section.

\section{``Stable and \texorpdfstring{$E$}{E}--familiar'' versus ``Stably \texorpdfstring{$E$}{E}--familiar''}\label{sec:versus}

We are now going to investigate the difference between $E$--familiar model categories that are also stable and stably $E$--familiar model categories. 
As a reminder, an $E$--familiar model category $\C$ is a model category where all cosimplicial frames
\[
\sset \lradjunction \C
\]
factor over $E$--local simplicial sets
\[
L_E \sset \lradjunction \C.
\]
A stably $E$--familiar model category is a model category where all stable frames
\[
\mathcal{S} \lradjunction \C
\]
factor over $E$--local sequential spectra
\[
L_E \mathcal{S} \lradjunction \C.
\]

Unfortunately, those two notions are not equivalent. We saw that a stably $E$--familiar model category is also $E$--familiar in Lemma \ref{lem:evaluation}, it is stable by definition. However, the converse is not true. The difference can be seen in the mapping spectra. We saw in Theorem \ref{localmappingspectrum} that a model category is stably $E$--familiar if and only if its mapping spectra are $E$--local. If the model category is only $E$--familiar and not stably $E$--familiar it only implies that the level spaces of each mapping spectrum are $E$--local. Although it also implies that the structure maps in the mapping spectra are weak equivalences, this is not enough to deduce that a spectrum is $E$--local. For example, it does not hold for $E = \h \mathbb{Z}/p$ as the colimit of $p$--complete groups is not necessarily $p$--complete.

\bigskip
Applying \cite{Hov01general}, or 
\cite{Sch97} to the model category $L_E \sset$ we obtain 
a model category of sequential spectra in $L_E \sset$. This model structure is denoted $\Sp(L_E \sset)$. The key to defining this model structure is the functor 
\[
Q: \Sp \longrightarrow \Sp
\]
which is the composition of a levelwise $E$--fibrant replacement functor and a fibrant replacement functor of sequential spectra.
Recall from \cite{Sch97} that the levelwise $E$--fibrant replacement functor 
$R_l$, is defined on a spectrum $X$ as follows:
$(R_l X)_0$ is the $E$--fibrant replacement of $X_0$
in $L_E \sset$. Then one considers the factorisation of 
the trivial map 
\[
X_k \coprod_{\Sigma X_{k-1}} \Sigma (R_l X_{k-1}) 
\overset{\sim}{\rightarrowtail} (R_l X)_k
\twoheadrightarrow *
\]
to obtain an $E$--fibrant space $(R_l X)_k$
with a levelwise $E$--equivalence $\eta_X \co X \to R_l X$. 

\begin{definition}
A map $f$ of spectra is a \textbf{$Q$--equivalence} if and only if $Qf$ is a levelwise weak equivalence in $\sset$.
\end{definition}

It is not hard to see that the class of $Q$--equivalences is the class of maps $f$ 
such that $R_l f$ is a $\pi_*$--isomorphism of spectra.

\begin{definition}
The \textbf{model category of spectra in $L_E \sset$}, $\Sp(L_E \sset)$,
is a model structure on the category $\Sp$ defined as follows. 
\begin{itemize}
\item Weak equivalences are the $Q$--equivalences.
\item Cofibrations are the cofibrations of $\Sp$. 
\item Fibrations are those maps that have the RLP with respect to cofibrations that are also $Q$--equivalences. 
\end{itemize}
\end{definition}

By \cite[Theorem 8.11]{Hov01general}, this defines a model structure.
Its fibrant objects, known as \textbf{$U$--spectra},
are the spectra whose spaces are $E$--fibrant and whose structure maps are 
weak equivalences of simplicial sets.

\begin{lemma}
There is a Quillen pair between spectra in $L_E \sset$ 
and $E$--local spectra. 
\[
\id : \Sp(L_E \sset)  
\,\,\raisebox{-0.1\height}{$\overrightarrow{\longleftarrow}$}\,\, 
L_E \Sp : \id.
\]
\end{lemma}
\begin{proof}
The cofibrations are the same for both model categories. 
We now show that a $Q$--equivalence is an $E$--equivalence. 
If $f \co X \to Y$ is a $Q$--equivalence, then $R_l f$ is a $\pi_*$--isomorphism, 
hence $R_l f$ is an $E$--equivalence. Now consider the following commutative diagram.
\[
\xymatrix{ \eta_X: X \ar[r]\ar[d]_{f}& \ar[d]^{R_l f} R_l X \\
\eta_Y: Y \ar[r] & R_l Y }
\]
The maps $\eta_X$ and $\eta_Y$ are levelwise $E$--equivalences, 
so they are also $E$--equivalences. Thus $f$ must also be a $E$--equivalence
by the two-out-of-three property.
\end{proof}

\begin{lemma}
Let $\C$ be an $E$--familiar and stable model category, and let $X$ be a cofibrant and fibrant object of $\C$. Then the Quillen pair
\[
X \wedge -: \mathcal{S} \,\,\raisebox{-0.1\height}{$\overrightarrow{\longleftarrow}$}\,\, \C: \Map(X,-)
\]
resulting from stable frames gives a Quillen pair 
\[
X \wedge -: \Sp(L_E \sset) \,\,\raisebox{-0.1\height}{$\overrightarrow{\longleftarrow}$}\,\, \C: \Map(X,-) .
\]
\end{lemma}
\begin{proof}
We show that $\Map(X,-)$ is a right Quillen functor from
$\C$ to $L_E \sset$. We know that it preserves trivial fibrations by adjunction as $\mathcal{S}$ and $\Sp(L_E \sset)$ have the same cofibrations. 
Thus we are left with showing that $\Map(X,-)$ preserves fibrations. We follow the proof
of \cite[Proposition 3.2]{Roi07}. Recall from 
\cite[Corollary 6.2]{Dug01} that it suffices to show that 
$\Map(X,-)$ takes fibrations between fibrant objects of $\C$
to fibrations of $\Sp(L_E \sset)$. 

\medskip
First we note that $\Map(X,-)$ takes fibrations of $\C$ to levelwise fibrations 
of $\Sp(L_E \sset)$: the $n^{th}$ level space of the spectrum $\Map(X,Y)$ is given by $\Map(X_n,Y)$, where $X_n \in \C^\Delta$ is the cosimplicial set representing the adjunction 
\[
X \wedge F_n(-): \sset \lradjunction \C: \Map(X,-)_n. 
\]
The model category $\C$ is $E$--familiar, so $$X_n \wedge -: \sset \lradjunction \C: \Map(X_n,-)$$ factors over $E$--local simplicial sets by assumption. In particular, $\Map(X_n,-)$ preserves fibrations. Hence $\Map(X,-)$ sends fibrations to level fibrations in $\Sp(L_E \sset)$. 

\medskip
Secondly, for fibrant $Y$, $\Map(X,Y)$ is 
an $\Omega$--spectrum, since $\Map(X,-)$ is a right Quillen functor
from $\C$ to $\Sp$. Thirdly a levelwise fibration between fibrant objects of 
$\Sp(L_E \sset)$ is a fibration. (For this statement, follow the proof of \cite[Proposition 3.2]{Roi07}, remembering that a $Q$--equivalence between $U$--spectra is a $\pi_*$--isomorphism.)

\medskip
Combining these three points we see that for fibrant $Y \in \C$, 
$\Map(X,Y)$ is a $U$--spectrum. So if $f \co Y \to Z$ 
is a fibration between fibrant objects of $\C$, 
then $\Map(X,f)$ is a levelwise fibration
between fibrant objects of $\Sp(L_E \sset)$ and the result follows.
\end{proof}

\begin{rmk}
With the same method as in the previous section, we could now also show the following: If $\C$ is an $E$--familiar stable model category then 
$\Ho(\C)$ is a closed $\Ho(\Sp(L_E \sset))$--module category. However, since not much is known about
the category  $\Ho(\Sp(L_E \sset))$, we would rather concentrate on investigating the case of $\Ho(L_E \Sp)$--module categories. 
\end{rmk}

We now give an example where an $E$--familiar and stable model category $\C$
is stably $E$--familiar. 

\begin{lemma}\label{lem:HRfamiliarity}
Let $R$ be a subring of the rationals. Then 
the stable model category of spectra in $\h R$--local
simplicial sets, $\Sp(L_{\h R} \sset)$, 
is the same as the stable model category of $\h R$--local spectra, $L_{\h R} \Sp$. 
\end{lemma}

\begin{proof}
We know that these categories have the same cofibrations
and that in each case a weak equivalence between fibrant objects is 
a $\pi_*$--isomorphism of spectra in simplicial sets.
If we can show that they have the same fibrant objects, then it follows
that the weak equivalences are the same. 

By \cite[Lemma 4.1]{SchShi02} a fibrant object of 
$L_{\h R} \Sp$ is an $\Omega$--spectrum
whose homotopy groups are $R$--local. The fibrant objects of 
$\Sp(L_{\h R} \sset)$ are the $U$--spectra, i.e. $\Omega$--spectra
where every level is a $\h R$--fibrant simplicial set. 
If a spectrum $X$ is fibrant in $L_{\h R}\Sp$ 
then each space must be $\h R$--local, hence $X$
is also a $U$--spectrum. 
Conversely, for large $n$, the $n^{th}$ homotopy group of an $\h R$--local space 
is $R$--local. Hence the homotopy groups of a $U$--spectrum 
are $R$--local. Thus any $U$--spectrum
is fibrant in $L_{\h R}\Sp$.
\end{proof}

\begin{corollary}
A model category $\C$ is stably $\h R$--familiar if and only if
it is $\h R$--familiar and stable. 
\end{corollary}

\section{Examples and Applications}\label{sec:applications}

We dedicate the final section of this paper to examples and applications of the technical work done in the previous sections. We will see how we can use $E$--local stable framings in the context of rigidity in the sense of \cite{Sch07} and \cite{Roi07}. Then another application will study stably $E$--familiar model categories in terms of an action of the stable homotopy groups of the $E$--local spheres. Finally, we can use all of this to classify algebraic $E$--familiar model categories. 
Let us start with some immediate consequences.

\bigskip
Some homology theories that are of crucial importance to stable homotopy theory are the chromatic Johnson--Wilson theories $E(n)$ with
\[
E(n)_* \cong \mathbb{Z}_{(p)}[v_1, v_2, ... , v_n, v_n^{-1}], \,\,\, |v_i|=2p^i-2
\]
as well as the Morava $K$--theories $K(n)$ with
\[
K(n)_* \cong \mathbb{Z}/p [v_n, v_n^{-1}].
\]
Note that the prime $p$ is absent from notation, and by convention, $E(0)=K(0)=\h\mathbb{Q}$. These homology theories and their Bousfield localisations provide important structural information about the stable homotopy category. For example, they are linked with periodicity and nilpotency phenomena. Also, the ``chromatic convergence'' theorem says that for a fixed prime $p$, $\Ho(L_{E(n)}\Sp)$ gives a better and better approximation of the stable homotopy category as $n$ increases. The ``thick subcategory theorem'' says that the $\Ho(L_{K(n)}\Sp)$ are the ``atomic'' localisations of the stable homotopy category. Finally, there is the chromatic pullback square linking the $E(n)$ with the $K(n)$. Details can be found in \cite{Rav92}. It is worth noting that $E(1)$ is the Adams summand of $p$--local complex $K$--theory, so localising with respect to $E(1)$ agrees with $p$--local $K$--localisation.

\bigskip
As there are plenty of known results about the relations between the $E(n)$ and $K(n)$ (see \cite{Rav84}), we can easily draw some first conclusions. For example, for a spectrum $X$ one has
\[
L_{K(n)}L_{E(n-1)}X \simeq *.
\]
Thus, a stably $E(n-1)$--familiar model category cannot be stably $K(n)$--familiar. One can also see that ``stably $E(n-1)$--familiar'' also implies ``stably $E(n)$--familiar''. Also, any stably $K(n)$--familiar model category is also stably $E(n)$--familiar. 

\subsection{Rigidity questions}\label{rigidity}

In recent years, Schwede showed that the stable homotopy category is homotopically determined by its triangulated structure only -- every stable model category $\C$ with $\Ho(\C)$ triangulated equivalent to $\Ho(\Sp)$ is automatically Quillen equivalent to $\Sp$, see \cite{Sch07}. Of course, this started the question of which other stable model categories are ``rigid'' in this sense.

\bigskip
{\bf Rigidity question} Let $\C$ be a stable model category. Assuming that there is an equivalence of triangulated categories
\[
\Phi: \Ho(L_E \Sp) \stackrel{\sim}{\longrightarrow} \Ho(\C),
\]
are $L_E \Sp$ and $\C$ Quillen equivalent?

\bigskip
To gain knowledge about the deeper structure of the stable homotopy category, the second author started considering the rigidity of chromatic Bousfield localisations of the stable homotopy category. For $p=2$, the result was that the $E(1)$--local stable homotopy category is rigid \cite{Roi07}. On the other hand, Franke showed in \cite{Fra96} that for $p >2$, the $E(1)$--local stable homotopy category possesses at least one ``exotic model''. Although the statements of the results in \cite{Sch07} and \cite{Roi07} look similar, the computational methods employed are quite different. We are going to see how these results and some elements of their proofs fit into the framework of stable $E$--familiarity. We restrict ourselves to the case of smashing localisations $L_E$ to make sure that the $E$--local sphere is a compact generator \cite[Theorem 3.5.2]{HovPalStr97}.

\bigskip
So what obvious obstructions are there for a model category $\C$ to be Quillen equivalent to some $L_E \Sp$? For example, Lemma \ref{lem:Quilleninvariance} tells us that $\C$ has to be stably $E$--familiar. In the case of Schwede's proof for $E=\mathbb{S}$, this condition is trivial, but for other $E$ this becomes a highly complicated computation. The second author could attempt this for $E(1)$ as in this range the telescope conjecture holds, giving a computable criterion for when a spectrum is $E(1)$--local. 
The core computation (using specific relations in $\pi_*(L_{K_{(2)}}\mathbb{S}$) ) was to show that every mapping spectrum $\Map(X,Y)$ is $K_{(2)}$--local \cite[Lemma 3.3]{Roi07}. In our words, this showed that $\C$ is stably $K_{(2)}$--familiar. So in the case of $E=K_{(2)}$, being stably $E$--familiar actually only depended on the triangulated structure of $\Ho(\C)$, which cannot be expected in the general case.

\bigskip
Given a stable model category $\C$, the first big step towards a Quillen equivalence with $L_E \Sp$ is the construction of a Quillen functor.  We can specify what this functor has to be when localising at $E$ is smashing. 

\begin{lemma}\label{lem:uniquequillen}
The following are equivalent when localisation at $E$ is smashing.
\begin{itemize}
\item There is a Quillen equivalence $F: L_E \Sp \lradjunction \C: G$.
\item $\omega X \sframe - : L_E \Sp \lradjunction \C: \Map(\omega X,-)$ is a Quillen equivalence for $X$  a fibrant--cofibrant replacement of $F(L_E \mathbb{S})$ and $\omega X$ a stable frame on $X$.
\end{itemize}
\end{lemma}

\begin{proof}
If $\C$ is Quillen equivalent to $L_E \Sp$, it is automatically stably $E$--familiar by Lemma \ref{lem:Quilleninvariance}. Since $F$ and $\omega X \wedge -$ agree on the $E$--local sphere, their derived functors agree by Lemma \ref{Ederivedfunctor}. Hence one is a Quillen equivalence if and only if the other is. 
\end{proof}

Using this, we see that if one has a triangulated equivalence between $\Ho(L_E \Sp)$ and $\Ho(\C)$, that comes from a Quillen functor, then it is determined by the image of the $E$--local sphere. 

\begin{corollary}
If the triangulated equivalence $$\Phi: \Ho(L_E \Sp) \longrightarrow \Ho(\C)$$ is realised by a Quillen functor, then it is realised uniquely up to natural transformations that are objectwise weak equivalences.
\end{corollary}
\qed

\begin{corollary}
The Quillen self-equivalences $L_E \Sp \longrightarrow L_E \Sp$ correspond to the Picard group 
$\Pic(L_E \Sp)$. 
\end{corollary}

\begin{proof}
By Lemma \ref{lem:uniquequillen}, every Quillen equivalence is of the form $(\omega X \sframe -, \Map(\omega X,-))$ for $X \in L_E \Sp$ a fibrant and cofibrant spectrum. By the uniqueness of framings (Lemma \ref{Ederivedfunctor}), such an adjunction agrees with the Quillen pair $(X \spectrasmash-, \Map(X,-))$ from Example \ref{stableexample}. In the second pair, $\spectrasmash$ denotes the smash product of spectra. Hence, $(X \spectrasmash -, \Map(X,-))$ provides a Quillen equivalence if and only if $X \in \Pic(L_E \Sp)$. 
\end{proof}

With the results of Section \ref{sec:Efamiliarstable}, we can reduce the question of whether a functor is a Quillen equivalence to studying a mapping spectrum, a technique related to Morita theory.

\begin{proposition}
A stable model category $\C$ and $L_E \Sp$ are Quillen equivalent if and only if the map
\[
L_E \mathbb{S} \longrightarrow R\Map(X,X)
\]
is a $\pi_*$--isomorphism for $X$ a fibrant--cofibrant compact generator of $\C$.
\end{proposition}

\begin{proof}
The ``only if'' direction is immediate. 
We note that as $X$ is a compact generator of $\C$, the functor $R\Map(X,-)$ reflects isomorphisms. Then $X \spectrasmash^L_E -$ and $R\Map(X,-)$ are inverse equivalences of categories if and only if the map
\[
Y \longrightarrow R\Map(X,X \spectrasmash^L Y)
\]
is an isomorphism in $\Ho(L_E \Sp)$ for all $Y$. Now for a compactly generated triangulated category, checking this is a standard argument \cite[Theorem 4.2]{Roi07}. The full subcategory of those $Y$ such that the above map is a weak equivalence is closed under exact triangles and coproducts. It contains the sphere by assumption. But any full subcategory of $\Ho(L_E \Sp)$, that contains the sphere and is closed under coproducts and exact triangles, is $\Ho(L_E \Sp)$ itself. This means that the above map is an isomorphism for all $Y$, which is what we wanted to prove.

\end{proof}

By adjunction, the above condition is equivalent to showing that 
\begin{equation}\label{exactequation}
X \sframe^L-: [\mathbb{S}, \mathbb{S}]^{L_E \Sp} \longrightarrow [X,X]^{\C}
\end{equation}
is an isomorphism, cf. Proposition \ref{schshieq} in the next subsection. Both Schwede in the case of $E=\mathbb{S}$, and the second author for $E=E(1)$, proved this by exploiting the relations in $\pi_*(L_E \mathbb{S})$. Using induction, Schwede reduces the question to elements in $\pi_*(\mathbb{S})$ that have Adams filtration one \cite{Sch01}. For odd primes, this is just $\alpha_1 \in \pi_{2p-3}(\mathbb{S})$. Schwede shows that $X \spectrasmash^L\alpha_1 \neq 0$ using extended powers of the mod--$p$ Moore spectrum. 

For $p=2$, the elements of Adams filtration 1 are the Hopf maps $\eta$, $\nu$ and $\sigma$. As $\nu$ and $\sigma$ can be constructed from $\eta$ using Toda bracket relations (which are preserved under exact functors), it can be reduced further to studying $X \spectrasmash^L \eta$ only. Multiplication by 2 on the mod--$2$ Moore spectrum $M$ is not only nonzero but also factors over $\eta$. Since this information is also preserved by exact functors, it can be deduced that
\[
X \spectrasmash^L \eta \neq 0 \,\,\,\mbox{in}\,\,\, [X,X]^{\C} \cong \mathbb{Z}/2
\]
for any $\C$ with $\Ho(\C) \simeq \Ho(\Sp_{(2)})$. Together with the inductive argument this shows that (\ref{exactequation}) is an isomorphism.

\bigskip
For $E=E(1)$ and $p=2$, the question is also reduced to the behaviour of $\eta \in \pi_*(L_{E(1)}\mathbb{S})$. But rather than using Adams filtration, the reduction exploits $v_1$--periodicity in the mod--$2$ homotopy groups. 

For $p>2$ the question is again reduced to $\alpha_1$. But since for odd primes $\pi_*(L_{E(1)}\mathbb{S})$ is not equipped with the same density of relations as for $p=2$, $X \spectrasmash^L \alpha_1 =0$ is possible, allowing space for exotic models \cite[Theorem 6.8]{Roi07}.

\bigskip
To summarize: when $L_E$ is smashing, $\C$ is Quillen equivalent to $\Ho(L_E \Sp)$ if and only if
\begin{itemize}
\item $\C$ is stably $E$--familiar and
\item $\mathbb{S} \longrightarrow R\Map(X,X)$ is an $E$--equivalence.
\end{itemize}
If these properties only depend on the triangulated structure of $\Ho(\C) \simeq \Ho(L_E \Sp)$, then $\Ho(L_E \Sp)$ is rigid.

\subsection{Modular Rigidity}

We can also use stable frames to look at the rigidity question from a different angle. How much homotopical information can be seen by the $\Ho(\Sp)$--module structure coming from stable frames? The answer is that for $E$--local spectra, the module structure encodes all relevant information. One could say that ``$\Ho(L_E \Sp)$ is rigid as a $\Ho(\Sp)$--module category''. 

\begin{theorem}\label{thm:modrigid}
Let $L_E$ be a smashing localisation and
\[
\Phi: \Ho(L_E \Sp) \longrightarrow \Ho(\C)
\]
be an equivalence of triangulated categories. Then the following are equivalent.
\begin{itemize}
\item $\Phi$ is the derived functor of a Quillen equivalence.
\item $\Phi$ is a $\Ho(\Sp)$--module functor.
\end{itemize}
\end{theorem}

\begin{proof}
We know that Quillen functors induce $\Ho(\Sp)$--module functors. For the other implication, assume that we have an equivalence of triangulated categories
\[
\Phi: \Ho(L_E \Sp) \longrightarrow \Ho(\C).
\]
By \cite[Theorem 3.1.1]{SchShi03}, $\C$ is Quillen equivalent to a category of module spectra over a ring spectrum $R$. Its Bousfield localisation $L_E R$ is still a ring spectrum. Let us now consider the category of $L_E R$--module spectra. We arrive at the following situation.
\[
\xymatrix{ \Ho(L_E \Sp) \ar[rr]^{\Phi} \ar[rrd]_{\Phi'} & & \ar[d]^{L_E} \Ho(R\mbox{--mod}) \\
 & & \Ho(L_E R\mbox{--mod}) 
}
\]
By \cite[Proposition VIII.3.2]{EKMM},
\[
\Ho(L_E R\mbox{--mod}) \cong \Ho(R\mbox{--mod})[E^{-1}],
\]
i.e. $\Ho(L_E R\mbox{--mod})$ is $\Ho(R\mbox{--mod})$ with the $E_*$--isomorphisms formally inverted. Because $\Phi$ is assumed to be a $\Ho(\Sp)$--module functor, it sends $E_*$--isomorphisms to $E_*$--isomorphisms, as does its inverse $\Phi^{-1}$. It follows that $\Phi^{-1}$ must factor over $\Ho(L_E R\mbox{--mod})$, giving an inverse to $\Phi'$. So now we know that if $\Phi$ is a $\Ho(\Sp)$--module equivalence, then it also induces an equivalence
\[
\Ho(L_E \Sp)= \Ho(L_E \Sp)[E^{-1}] \longrightarrow \Ho(R\mbox{--mod})[E^{-1}].
\]

Consequently, $\Phi'$ is a triangulated equivalence, and so is
\[
L_E: \Ho(R\mbox{--mod}) \longrightarrow \Ho(L_E R\mbox{--mod}).
\]
So $\C$, $R\mbox{--mod}$ and $L_E R\mbox{--mod}$ are all Quillen equivalent. Since $L_E R$ is obviously $E$--local and $L_E$ is smashing, the category $L_E R\mbox{--mod}$ is an $L_E \Sp$--model category. Thus, $\C$ is stably $E$--familiar.

As $\C$ is stably $E$--familiar, we can now consider Quillen functors
\[
X \sframe -: L_E \Sp \longrightarrow \C \,\,\,\,\,\mbox{for}\,\, X \in \C.
\]
Take $X = \Phi(L_E \mathbb{S})$. Remember that $\Phi$ is a $\Ho(\Sp)$--module functor and that $X \sframe^L - = X \sframe^L_{E} -$. Then
\[
X \sframe^L_{E} - = \Phi(L_n \mathbb{S}) \sframe^L_{E} - = \Phi( L_n \mathbb{S} \spectrasmash^L - ) = \Phi(-).
\]
This means that $X \sframe -$ is a Quillen functor with left derived functor $\Phi$, which is what we wanted to prove.

\end{proof}

This means that for $L_E$ smashing, all higher homotopy information of $L_E \Sp$ is encoded in the stable frames.

\subsection{Linearity and Uniqueness}\label{sec:linearity}

A major application of framings is using them to define an action of the stable homotopy groups of spheres on the morphisms groups of the homotopy category of a stable model category. In \cite{SchShi02}, Schwede and Shipley define this $\pi_* (\mathbb{S})$--action and show how it can be used to examine whether a stable model category is Quillen equivalent to the category of spectra. 

We are going to use our work on $E$--local framings to investigate whether a stably $E$--familiar model category $\C$ is Quillen equivalent to the category of $E$--local spectra $L_E \mathcal{S}$,
in the case that localisation at $E$ is smashing. 
There, the action of $\pi_* \mathbb{S}$ passes through $\pi_*(L_E \mathbb{S})$, which is an advantage, as in many cases $\pi_* L_E \mathbb{S}$ is better understood, more computable and more highly structured than $\pi_* \mathbb{S}$.

\bigskip
First of all, let $R_*$ be a graded ring. We say that a triangulated category $\mathcal{T}$ is \textbf{$R_*$--linear} if $\mathcal{T}$ has an action of $R_*$ which is compatible with the triangulated structure, i.e. there are bilinear pairings
\[
R_n \sframe \mathcal{T}(X,Y) \longrightarrow \mathcal{T}(X[n],Y)
\]
for all $X,Y \in \mathcal{T}$ which are unital, associative, central and compatible with the shift in $\mathcal{T}$ \cite[Definition 2.2]{SchShi02}.

An \textbf{$R_*$--exact functor} is a functor of triangulated categories which is compatible with the $R_*$--action, see \cite[Definition 2.2]{SchShi02}.

\begin{ex}
Let $\mathcal{M}$ be a monoidal triangulated category with unit $I$ and $\mathcal{T}$ a module over this category. 
Then we see that the module action makes every group $\mathcal{T}(A,B)$ into a $\mathcal{M}(I,I)$--linear category. 

For $\mathcal{T}=\Ho(\C)$ and $\mathcal{M}=\Ho(\Sp)$, this recovers \cite[Construction 2.4]{SchShi02}. If $\C$ is a stably $E$--familiar model category, the above construction makes $\Ho(\C)$ into a $\pi_*(L_E \mathbb{S})$--linear category.
Furthermore, the ring $\pi_*(\mathbb{S})$ acts on $[X,Y]_*^{\C}$ via the localisation map
$\pi_*(\mathbb{S}) \to \pi_*(L_E \mathbb{S}).$
\end{ex}

\begin{rmk}
One might want to study stable model categories
whose homotopy categories are
$\pi_*(L_E \mathbb{S})$--triangulated, 
but this notion has a difficulty. Let $\C$ be such a model category, 
then one would want the action of 
$\pi_*(L_E \mathbb{S})$ on homotopy classes of maps (coming from the 
$\pi_*(L_E \mathbb{S})$--triangulation) to 
be related to the map 
$$[\mathbb{S}, L_E \mathbb{S}]^{\Sp}_* \times [X,Y]_*^{\C}
\to [X,Y \sframe^L L_E \mathbb{S}]_*^{\C}
$$
that comes from stable framings. The only way of achieving a 
suitably useful relation seems to be requiring that for any $Y$ the map 
$Y \to Y \sframe^L L_E \mathbb{S}$
is a weak equivalence. In the smashing case, which is the one of greatest interest, this is precisely the
condition that $\C$ be stably $E$--familiar. 
\end{rmk}

\begin{lemma}\label{lem:localmodfunctors}
A Quillen pair between stably $E$--familiar model categories induces 
an adjunction of closed $\Ho (L_E \Sp)$--modules
on homotopy categories. 
\end{lemma}
\begin{proof}
Take a Quillen pair between stably $E$--familiar model categories 
$$
F : \mathcal{C} \lradjunction \mathcal{D} : G
$$
then the categories $\Ho (\mathcal{C})$ and $\Ho(\mathcal{D})$
are closed $\Ho (L_E \Sp)$--modules by Theorem \ref{thm:EfamiliarEmodule}. 
We want to show that the derived adjunction
$$
LF : \Ho(\mathcal{C}) \lradjunction \Ho(\mathcal{D}) : RG
$$
is an adjunction of closed $\Ho (L_E \Sp)$--modules. 
By \cite[Definition 4.1.14]{Hov99}, this amounts to showing that 
$LF$ is a $\Ho(L_E \Sp)$--module functor. 
So we need a natural isomorphism, in $\Ho(\mathcal{D})$
$$
m \co LF(X) \Esframe K \to LF(X \Esframe K)
$$
for any $X \in \Ho(C)$ and $K \in \Ho(L_E \Sp)$, 
which satisfies associativity and unital coherence conditions. 
By \cite[Theorem 7.3]{Len11}, the functor $LF$ is a 
$\Ho(\Sp)$--module functor. Let 
$$
m' \co LF(X) \sframe^L K \to LF(X \sframe^L K)
$$
be the associativity isomorphism of this structure. 
We can choose $m'$ to be equal to $m$, since the cofibrant replacement
functor of $L_E \Sp$ can be chosen to agree with that of $\Sp$. 
It remains to show that this $m$ is natural on the category
$\Ho(L_E \Sp)$. Take an $E$--equivalence
$f \co L \to K$, and consider the diagram 
$$
\xymatrix@C+1cm{
LF(X) \Esframe L 
\ar[r]^{1 \Esframe f} \ar[d]^{m'} & 
LF(X) \Esframe K
\ar[d]^{m'} \\
LF(X \Esframe L) 
\ar[r]^{LF(1 \Esframe f)} & 
LF(X \Esframe K).
}
$$
By naturality of $m'$ on $\Ho(\Sp)$ this diagram commutes.
The top and bottom horizontal maps are isomorphisms, so the analogous diagram
involving $f^{-1}$ commutes. This shows that $m=m'$ is a natural isomorphism
on $\Ho(L_E \Sp)$. 
The coherence conditions follow immediately. 
\end{proof}

\begin{corollary}\label{cor:linearityfactors}
A Quillen pair between stably $E$--familiar model categories is
$\pi_*(L_E \mathbb{S})$--linear.
\end{corollary}
\qed

We can now use this $\pi_*(L_E \mathbb{S})$--action to study whether a stably $E$--familiar stable model category $\C$ is Quillen equivalent to $L_E \mathcal{S}$. For this we need to restrict ourselves to smashing localisations in order to guarantee that the $E$--local sphere is still a small weak generator \cite[Theorem 3.5.2]{HovPalStr97}.

We can easily state an $E$--local version of \cite[Theorem 5.3]{SchShi02}. The proof is obviously going to be extremely similar to the original, so we omit it and refer to Schwede's and Shipley's version. The only difference being using the fact that in a stably $E$--familiar model category, mapping spectra are $E$--local. 

\begin{proposition}\label{schshieq}
Let $\C$ be a stably $E$--familiar model category. Then the following are equivalent.
\begin{enumerate}
\item There is a chain of Quillen equivalences between $\C$ and $L_E \mathcal{S}$.
\item There exists a $\pi_*(L_E \mathbb{S})$--linear equivalence between $\Ho(\C)$ and $\Ho(L_E \mathcal{S})$.
\item The model category $\C$ has a small weak generator $X$ for which $[X,X]^\C_*$ is freely generated as a $\pi_*(L_E \mathbb{S})$--module by the identity of $X$.
\item The homotopy category $\Ho(\C)$ has a cofibrant--fibrant small weak generator $X$ for which $L_E \mathbb{S} \longrightarrow \Map(X,X)$ is a weak equivalence of spectra.
\end{enumerate}
Furthermore if $X$ is a cofibrant and fibrant object of $\C$ which satisfies 
either of the last two conditions then the adjunction
$(X \sframe -, \Map(X,-)$ is a Quillen equivalence between $\C$ and
$L_E \Sp$. 
\end{proposition}
\qed

\subsection{Algebraic model categories}
Another interesting class of model categories to consider is the class of \textbf{algebraic model categories}. An algebraic model category is a $\ch(\mathbb{Z})$--model category, where $\ch(\mathbb{Z})$ denotes the model category of chain complexes of abelian groups. Sometimes this is also called a \textbf{dg--model category}, cf. \cite{SchShi03}. We would like to investigate what stably $E$--familiar algebraic model categories look like. The mapping spectra of algebraic model categories carry some special structure: they are products of Eilenberg--MacLane spectra. Together with some knowledge of Bousfield localisations of Eilenberg--MacLane spectra we can draw some interesting conclusions. 

\begin{lemma}\label{lem:EMLmapping}
Let $\C$ be an algebraic model category. Then for each $X, Y \in \C$, the mapping spectrum $\Map(X,Y)$ is a product of Eilenberg--MacLane spectra.
\end{lemma}

\begin{proof}
Because of the enrichment over chain complexes, $\Map(X,Y)$ is not only a spectrum of simplicial sets, but also a spectrum of simplicial abelian groups. It is known that these are products of Eilenberg--MacLane spectra, see e.g. \cite[Proposition III.2.20]{GJ99} or \cite[Section 2.6]{DugShi07}. More specifically,
\[
\Map(X,Y) \simeq \prod\limits_{n \ge 0} \h([X,Y]^\C_n) = \prod\limits_{n\ge 0} \h(H_*(\C(X,Y))).
\]
\end{proof}

In recent work, Guti{\'e}rrez computed the Bousfield localisation of Eilenberg MacLane--spectra with respect to important homology theories $E$ \cite{Gut10}.

\begin{theorem}[Guti{\'e}rrez]
Let $G$ be an abelian group. Then 
\begin{itemize}
\item $L_{K(n)} \h G \simeq * $ for $n \ge 1$
\item $L_{E(n)} \h G \simeq L_{\h\mathbb{Q}} \h G$ for all $n$
\end{itemize}
\end{theorem}

\begin{corollary}
There are no algebraic stably $K(n)$--familiar model categories for $n \ge 1$.
\end{corollary}
\qed

\begin{corollary}
Let $\C$ be an algebraic model category. Then $\C$ is stably $E(n)$--familiar if and only if $\C$ is $\h\mathbb{Q}$--familiar, i.e. rational. 
\end{corollary}
\qed

From this we can conclude immediately that $L_{E(n)}\Sp$ and $L_{K(n)}\Sp$ are not algebraic for $n \ge 1$. But the computations of Guti{\'e}rrez reach even further, allowing us to classify algebraic $L_E \Sp$ for all $E$.

\begin{theorem}\label{thm:algrational}
The category of $E$--local spectra $L_E \Sp$ is algebraic if and only if $L_E=L_{\h\mathbb{Q}}$.
\end{theorem}

\begin{proof}
Let $\alpha \in \pi_*(L_E \mathbb{S})$, $X,Y \in L_E \Sp$. Then, by adjunction, the following diagram commutes.
\[
\xymatrix{ [\mathbb{S},L_E R\Map(X,Y)]^{\Sp} \ar[r]^{\cong} \ar[d]_{-\circ\alpha}& [\mathbb{S}, R\Map(X,Y)]^{L_E\Sp} \ar[r]^{\,\,\,\,\,\,\,\cong}_{adj}\ar[d]_{-\circ\alpha}& [X,Y]^{L_E \Sp} \ar[d]_{\alpha \sframe^L -}\\
[\mathbb{S},L_E R\Map(X,Y)]^{\Sp} \ar[r]^{\cong} & [\mathbb{S}, R\Map(X,Y)]^{L_E\Sp} \ar[r]^{\,\,\,\,\,\,\,\cong}_{adj}& [X,Y]^{L_E \Sp}
}
\]
By Lemma \ref{lem:EMLmapping}, $R\Map(X,Y)$ is a product of Eilenberg--MacLane spectra. By \cite[Corollary 4.2]{Gut10}, the $E$--localisation of an Eilenberg--MacLane spectrum is again a product of Eilenberg--MacLane spectra. So by degree reasons, precomposition with $\alpha$ in the above diagram is trivial unless $\alpha$ is in degree zero.

In particular, this is true for $X=Y=L_E \mathbb{S}$. Remembering that the action of homotopy groups is unital, this implies that $[\mathbb{S},\mathbb{S}]^{L_E \Sp}$ is concentrated in degree zero only. This means that 
\[
L_E \mathbb{S} =\h R \,\,\,\,\mbox{for some} \,\,\,R.
\]
As localisation is idempotent, 
\[
L_E \h R = \h R.
\]

By \cite[Theorem 3.5]{Gut10} we further have
\[
L_E \h R = L_{\h G} \h R
\]
where $G$ is either $\mathbb{Z}/P$ or $\mathbb{Z}_{P}$ for some set of primes $P$. Thus, in our case localisation is either $P$--localisation or $P$--completion. But we know that the $P$--complete sphere can never have its homotopy concentrated in one degree. And it is the same for the $P$--local sphere unless $P$--completion is rationalisation, leaving us with the only possible case $R = \mathbb{Q}$. 

Now it is only left to prove that if $L_E \mathbb{S} = \h\mathbb{Q}$, then $L_E$ is rationalisation. We know that $S \longrightarrow \h\mathbb{Q}$ is an $E$--equivalence. Consequently, the cofibre of this map, $C$, is $E$--acyclic. If we consider the cofibre sequence
\[
M(n) \longrightarrow C \stackrel{\cdot n}{\longrightarrow} C
\]
where $M(n)$ is the mod--$n$ Moore spectrum, we see that $M(n)$ must also be $E$--acyclic. Putting this into the cofibre sequence
\[
E \stackrel{\cdot n}{\longrightarrow} E \longrightarrow E \spectrasmash M(n) \simeq *
\]
we see that multiplication by $n$ is an isomorphism on $E$ for all $n$, hence $E$ must be rational. 

If a spectrum $E$ is rational, then it is a module spectrum over the rational sphere $\h\mathbb{Q}$. However, module spectra over Eilenberg--MacLane spectra are again Eilenberg--MacLane spectra. This means that $E$ is a wedge of shifts of products of $\h\mathbb{Q}$ and consequently, $L_E=L_{\h\mathbb{Q}}$. 
\end{proof}

This result fits in nicely
with the long-known result that $\Ho (L_{\h \mathbb{Q}} \Sp)$ is equivalent to the
derived category of rational chain complexes. This statement was improved in 
\cite{Shi07} which proves that $L_{\h \mathbb{Q}} \Sp$ is Quillen equivalent to 
 $\ch(\mathbb{Q})$. 

\bigskip
We now have a good understanding of how $\Ho(L_E \Sp)$ can act on the homotopy category of
a stable model category and we have related this to actions of $\Ho(L_E \sset)$.
When $L_E$ is smashing we applied this knowledge to questions of rigidity and proven a 
uniqueness statement for $\Ho (L_E \Sp)$. Finally we have studied actions of 
$\Ho(L_E \Sp)$ on algebraic model categories and seen that this action always passes
through $\Ho(L_{\h \mathbb{Q}} \Sp)$.

\end{document}